\documentclass{amsart}
\usepackage{latexsym,amssymb,graphicx,amscd,amsmath, url}
\usepackage{graphicx}
\usepackage{fancyhdr}

\newtheorem{thm}{Theorem}[section]
\newtheorem{lem}[thm]{Lemma}
\newtheorem{prop}[thm]{Proposition}
\newtheorem{cor}[thm]{Corollary}
\newtheorem{de}[thm]{Definition}
\newtheorem{rem}[thm]{Remark}

\newtheorem{note}[thm]{Notation}

\newcommand{\bz}{{\mathbb{Z}}}

\newcommand{\Tet}{{\operatorname{Tet}}}

\newcommand{\TV}{{\operatorname{TV}}}
\newcommand{\Trace}{{\operatorname{Trace}}}
\newcommand{\rank}{{\operatorname{Rank}}}

\newcommand{\vs}{{V(S)}}
\newcommand{\vsjn}{{V(S_1(j_1,\cdots,j_n))}}
\newcommand{\vsim}{{V(S_0(i_1,\cdots,i_m))}}
\newcommand{\vsn}{{V(S_1)}}
\newcommand{\vsm}{{V(S_0)}}
\newcommand{\tnm}{{T^n_m}}
\newcommand{\ztnm}{{Z(T^n_m)}}

\newcommand{\tnmij}{{Z(T^{n,(j_1,...,j_n)}_{m,(i_1,...,i_m)})}}
\newcommand{\sti}{(S^2\times I,0)}
\newcommand{\stl}{S^2\times\{1\}}
\newcommand{\sto}{S^2\times\{0\}}
\newcommand{\sts}{S^2\times S^1}
\newcommand{\stp}{S^2\times\{p\}}
\newcommand{\stpo}{S^2\times\{p_0\}}
\newcommand{\stpl}{S^2\times\{p_1\}}

\begin{document}

\title[The Turaev-Viro Endomorphism and the Jones Polynomial]{On the Turaev-Viro Endomorphism, and the colored Jones polynomial}

\author{Xuanting Cai}
\address{Mathematics Department,
Louisiana State University,
Baton Rouge, Louisiana}
\email{xcai1@math.lsu.edu}
\urladdr{\url{www.math.lsu.edu/~xcai1}}

\author[Patrick Gilmer]{Patrick M. Gilmer}
\address{Mathematics Department,
Louisiana State University,
Baton Rouge, Louisiana}
\email{gilmer@math.lsu.edu}
\urladdr{\url{www.math.lsu.edu/~gilmer}}

\date{September 9, 2012}

\begin{abstract}
By applying a variant of the TQFT constructed by Blanchet, Habegger, Masbaum, and Vogel, and using a construction of Ohtsuki,
we define a module endomorphism for each knot $K$ by using a tangle obtained from a surgery presentation of $K$.
We show that it is strong shift equivalent  to the Turaev-Viro endomorphism associated to $K$.  Following Viro, we consider
the endomorphisms that one obtains after  coloring the meridian and longitude of the knot. 
We show that the traces of these endomorphisms encode the same information as the colored Jones polynomials of $K$ at a root of unity.
Most of the discussion is carried out in the more general setting of infinite cyclic covers of 3-manifolds.
\end{abstract}

\maketitle
\setcounter{tocdepth}{2}
\tableofcontents

\section{Introduction}
\label{introduction}

\subsection{History}
Walker first noticed \cite{Walker2} that the endomorphism induced in a $2+1$-TQFT (defined over a field) by the exterior of a closed off Seifert surface of a knot in zero-framed surgery along the knot can be used to give lower bounds for the genus of the knot. He did this by showing  the number of non-zero eigenvalues of this  endomorphism  counted with multiplicity is an invariant \cite{Walker2}, i.e. it does not depend on the choice of the Seifert surface. 
Thus the number of such eigenvalues must be less than or equal to the dimension of the vector space that the TQFT assigns to a closed surface of this minimal genus.

 Next Turaev and Viro \cite{TV}, again assuming the TQFT is defined over a field, saw that the similarity class of the induced map on the vector space associated to a Seifert surface modulo the generalized $0$-eigenspace was a stronger  invariant.  If the TQFT is defined over a more general commutative ring, the second author observed that the strong shift equivalence class of the   endomorphism is an invariant of the knot \cite{G3}.   Strong shift equivalence (abbreviated SSE) is a notion from symbolic dynamics which we will discuss in \S \ref{sse}  below. 
For a TQFT defined over a field $F$, the similarity class considered by Turaev-Viro is a complete invariant of SSE.
 In this case, the vector space  modulo the generalized $0$-eigenspace together with the induced automorphism, considered as a module over $F[t,t^{-1}]$, is called 
 the Turaev-Viro module. It should be considered as somewhat analogous to the Alexander module. The  order of  the Turaev-Viro module is called the Turaev-Viro polynomial and  lies in  $F[t,t^{-1}]$.  We will refer to the endomorphisms constructed as above (and those in the same SSE class) as   Turaev-Viro endomorphisms.

In  \cite{G1,G2}, Turaev-Viro endomorphisms were studied 
and  methods for computing the endomorphism explicitly were given. These methods adapted Rolfsen's surgery technique of studying infinite cyclic covers of knots.   
This method requires finding a surgery description of the knot; 
that is a framed link in the complement of the unknot such that the framed link describes $S^3$ and the unknot represents the original knot. 
Moreover each of the components of the framed link should have linking number zero with the unknot.  
For this method to work, it is important that the surgery presentation have a nice form. 
In this paper, we will  show that all knots have a surgery presentations  of this form (in fact an even nicer form that we will call standard.)  
Another explicit method of computation was  given by Achir, and Blanchet  \cite{AB}. 
This method starts with any Seifert surface.
The second author   also considered the further invariant obtained by decorating a knot with a colored meridian (this was needed to give  formulas for the Turaev-Viro endomorphism of a connected sum, and to use the Turaev-Viro endomorphism to compute the quantum invariants of branched cyclic covers of the knot).

 Ohtsuki \cite{O,O2}  arrived at the same invariant as the Turaev-Viro polynomial
 but from a very different point of view. Ohtsuki extracts this invariant from a surgery description of a knot (alternatively of a closed 3-manifold with  a primitive one dimensional cohomology class) and the data of a modular category.   His method starts from any surgery description standard or not. This is a  significant advantage of his approach.  Ohtsuki's proof of the invariance of the polynomial in \cite{O} is only sketched. He stated that his invariant is the same as the Turaev-Viro polynomial, but does not give an explanation.  

Recently Viro has returned to these ideas \cite{Viro',Viro}. He has studied the Turaev-Viro endomorphism of a knot after coloring both the meridian and  the longitude of the knot.  Viro observed that a weighted sum of the traces of these endomorphisms  is the colored Jones polynomial evaluated at a root of unity.  

In \cite{G1,G2,G3},  Turaev-Viro endomorphisms were defined more generally  for infinite cyclic cover of 3-manifolds.
Suppose $(M,\chi)$ is a closed connected oriented  $3$-manifold $M$ with $\chi\in H^1(M, \bz)$ such that $\chi:H_1(M,\bz)\rightarrow\bz$
is onto. Let $M_{\infty}$ be the infinite cyclic cover of $M$ corresponding to $\chi$.
Choose a surface $\Sigma$ in $M$ dual to $\chi$.
By lifting $\Sigma$ to $M_{\infty}$, we obtain a fundamental domain $E$ with respect to the action of $\bz$ on $M_{\infty}$.
$E$ is a cobordism from  a surface $\Sigma$ to itself.
Let $(V,Z)$ be a $2+1$-TQFT on the cobordism category of extended $3$-manifolds and extended surfaces.
Applying $(V,Z)$ to $E$ and $\Sigma$, 
we can construct an endomorphism $Z(E):V(\Sigma) \rightarrow V(\Sigma)$. 
In \cite{G3}, it is proved that the strong shift equivalent class of $Z(E):V(\Sigma) \rightarrow V(\Sigma)$ is an invariant of the pair $(M,\chi)$, i.e. it does not depend on the choice of $\Sigma$. 
We  denote this SSE class by
$\mathcal{Z}(M,\chi)$.
We will sometimes refer to a  pair $(M, \chi)$ as above, informally,  as a 3-manifold with an infinite cyclic covering. 

The  knot invariants discussed above can be obtained as special cases of the above invariants of 3-manifolds with an infinite cyclic covering.   For any {\it oriented} knot $K$ in $S^3$, 
we obtain an extended $3$-manifold $S^3(K)$ by doing $0$-surgery along $K$.
We choose $\chi$ to be the integral cohomology class that evaluates to $1$ on a positive meridian of $K$.  Then it is easy to see that the invariant $\mathcal{Z}(S^3(K),\chi)$ corresponding to $(S^3(K),\chi)$ only depends on $K$.   If our TQFT is defined for 3-manifolds with colored links, one may obtain further invariants by 
coloring the meridian and the longitude (a little further away) of the knot. 

For the knot invariants discussed above,  it is  required, in general, that $K$ be oriented \footnote{ For TQFTs over a field satisfying some common axioms, the Turaev-Viro endomorphisms of a knot and its inverse have the same SSE class. This follows from \cite[Proposition 1.5]{G1} and Proposition \ref{similar}.} This is so that  the exterior of a Seifert surface acquires a direction as a cobordism from the Seifert surface to itself. However, we decided to delay mentioning this technicality.
To avoid issues that arise from phase anomalies in TQFT, in this paper, we  work with extended manifolds as in Walker \cite{Walker} and Turaev \cite{T}.
In this introduction, we omit mention of the integer weights and lagrangian subspaces of  extended manifolds. We discuss extended manifolds carefully in the main text.

\subsection{Results of this paper}

Inspired by Ohtsuki, we  construct a SSE class ${Z}(M,\chi)$ from a 
framed ( or banded)  
tangle in $S^2 \times I$ that arises in a surgery presentation of $(M,\chi)$. We call this the tangle endomorphism.
Moreover we show that the endomorphism (or square matrix) that Ohtsuki  considers in this situation is well defined up to SSE.
By relating the definition of the Turaev-Viro endomorphism to Ohtsuki's matrix, we give a different proof of the invariance of Ohtsuki's invariant. In fact, we show that Ohtsuki's matrix has the same SSE class as the Turaev-Viro endomorphism, i.e.  $\mathcal{Z}(M,\chi)={Z}(M,\chi)$.  
We do not prove these results in the general case of a TQFT arising from a modular category. We only work in the context of the skein approach for TQFTs associated to  $SO(3)$ and  $SU(2)$.
We work with a modified Blanchet-Habegger-Masbaum-Vogel  approach \cite{BHMV} as outlined in \cite{GM}. This theory  is defined over a slightly localized cyclotomic ring of integers. It is worthwhile studying endomorphisms defined up to strong shift equivalence over this ring rather than passing to a field.

We show that the  traces of the Turaev-Viro polynomial of knots with the meridian and longitude colored turns out to encode exactly the same information as the colored Jones polynomial evaluated at a root of unity.

\subsection{Organization of this paper}
In section \ref{preliminary}, we discuss extended manifolds, a variant of the TQFT constructed in \cite{BHMV}, surgery presentations and 
the definition of SSE.  
In section \ref{thetanglemorphism}, we construct an endomorphism for each framed tangle in $S^2 \times I$ and apply it to the tangle obtained from a surgery presentation of an infinite cyclic cover of a 3-manifold. 
We call it the tangle endomorphism.
Then we state Theorem \ref{main} which states that  the SSE class of a tangle  endomorphism constructed  from a surgery presentation of $(M, \chi)$ is an of  invariant $(M, \chi)$.
In section \ref{TVM}, we discuss technical details concerning the Turaev-Viro endomorphism for $(M,\chi)$, and
 the method of calculating $\mathcal{Z}(M, \chi)$ introduced in \cite{G1}.
In section \ref{relation}, we relate the tangle endomorphism associated to  a nice surgery presentation to the corresponding Turaev-Viro endomorphism.
In section \ref{PMT}, we prove  Theorem \ref{main}. 
In section \ref{CJ}, we give  formulas relating the colored Jones  polynomial to the traces of Turaev-Viro endomorphism of a knot whose meridian and longitude are colored.
In section \ref{example}, we compute two examples to illustrate these ideas.

\subsection{Convention} All surfaces and 3-mainifolds are assumed to be oriented.


\section{Preliminaries}
\label{preliminary}

\subsection{Extended surfaces and extended $3$-manifolds.}
For each  integer $p \ge 3$, Blanchet, Habegger, Masbaum and Vogel define a TQFT from quantum invariants of $3$-manifolds at $2p$th root of unity over a $2+1$-cobordism category in \cite{BHMV}.
The cobordism category has surfaces with $p_1$-structures as objects and $3$-manifolds with $p_1$-structures as morphisms.
They introduce $p_1$-structures in order to resolve the framing anomaly.
Following \cite{G4,GM}, we will adapt the theory by using extended surfaces and extended $3$-manifolds in \cite{Walker,T} instead of $p_1$-structures to resolve the framing anomaly.
In the following, all homology groups have rational coefficients except otherwise stated.

\begin{de}
An extended surface $(\Sigma,\lambda(\Sigma))$ is a closed  surface $\Sigma$ with a lagrangian subspace $\lambda(\Sigma)$ of $H_1(\Sigma)$ with respect to its intersection form, 
which is a symplectic form on $H_1(\Sigma)$.
\end{de}

\begin{de}
An extended $3$-manifold $(M,r,\lambda(\partial M))$ is a $3$-manifold with an integer $r$, 
called its weight, 
and whose oriented boundary $\partial M$ is given an extended surface structure with lagrangian subspace $\lambda(\partial M)$. 
If $M$ is a closed extended 3-manifold, 
we may denote the extended 3-manifold simply by $(M,r)$.
\end{de}

\begin{rem}
Suppose we have an extended $3$-manifold $(M,r,\lambda(\partial M))$ and $\Sigma\subset\partial M$ is a closed surface.
Then
\begin{equation}
\lambda(\partial M)\cap H_1(\Sigma)
\notag
\end{equation}
need  not be a lagrangian subspace of $H_1(\Sigma)$.
\end{rem}

\begin{de}
Suppose we have an extended $3$-manifold $(M,r,\lambda(\partial M))$ and $\Sigma\subset\partial M$ is a closed surface.  If $\lambda(\partial M)\cap H_1(\Sigma)$ is a lagrangian subspace of $H_1(\Sigma)$,
 we call $\Sigma$ equipped with this lagrangrian a boundary surface of the extended 3-manifold $(M,r,\lambda(\partial M))$.   
\end{de}

\begin{note}
If $\Sigma$ is a surface,
we use $\bar{\Sigma}$ to denote the surface $\Sigma$ with the opposite orientation.
\end{note}

\begin{prop}
\label{lagrangiansubspace}
Suppose $(V_1,\omega_1)$ and $(V_2,\omega_2)$ are two symplectic vector spaces.
Consider the symplectic vector space $V_1\oplus V_2$ with symplectic form $\omega_1\oplus\omega_2$. 
We can identify $V_1$ and $V_2$ as symplectic subspaces of $V_1\oplus V_2$.
If $\lambda\subset V_1\oplus V_2$ is a lagrangian subspace such that $\lambda\cap V_1$ is a lagrangian subspace of $V_1$, 
then $\lambda\cap V_2$ is a lagrangian subspace of $V_2$.  
\end{prop}
\begin{proof}
Since $\lambda\cap V_1=\text{span}<a_1,\cdots,a_n>$ where $n=\frac{1}{2}\text{dim}(V_1)$,
we can assume that
\begin{equation}
\lambda=\text{span}<(a_1,0),\cdots,(a_n,0),(c_1,b_1),\cdots,(c_m,b_m)>,
\notag
\end{equation} 
where $m=\frac{1}{2}\text{dim}V_2$.
Since for any $i,j$
\begin{eqnarray}
0&=&\omega_1\oplus\omega_2((a_i,0),(c_j,b_j))\notag\\
&=&\omega_1(a_i,c_j)+\omega(0,b_j)\notag\\
&=&\omega_1(a_i,c_j),\notag
\end{eqnarray}
we have $c_j\in (\lambda\cap V_1)^\perp =\lambda\cap V_1$.
Therefore, 
\begin{equation}
\lambda=\text{span}<(a_1,0),\cdots,(a_n,0),(0,b_1),\cdots,(0,b_m)>.
\notag
\end{equation}
That means $\text{dim}(\lambda\cap V_2)=m$.
So $\lambda\cap V_2$ is a lagrangian subspace in $V_2$.
\end{proof}

\begin{cor}[\cite{GM}]
\label{complement}
Suppose we have an extended $3$-manifold $(M,r,\lambda(\partial M))$ and $\Sigma\subset\partial M$ is a boundary surface.
Then $\partial M-\Sigma$, equipped with the lagrangian $H_1(\partial M-\Sigma) \cap \lambda(\partial(M))$,
 is also a boundary surface.
\end{cor}
\begin{proof}
This follows from Proposition \ref{lagrangiansubspace}.
\end{proof}

In the next three definitions, we describe the morphisms and the composition of morphisms in  
$\mathcal{C}$, a cobordism category  whose objects are extended surfaces.

\begin{de}
Let $(M,r,\lambda(\partial M))$ be an extended $3$-manifold.
Suppose
\begin{equation}
\partial M=\bar{\Sigma}\cup\Sigma',
\notag
\end{equation}
and this boundary has been partitioned into two boundary surfaces $\bar{\Sigma}$, called  (minus) the source, and $\Sigma'$,  called the target.
We write
$$(M,r,\lambda(\partial M))
:(\Sigma,\lambda (\bar{\Sigma})) 
\rightarrow (\Sigma',\lambda (\Sigma')),$$
and call $(M,r,\lambda(\partial M))$ an extended cobordism. 
\end{de}

\begin{de}
Let $\Sigma$ be a boundary surface of an extended $3$-manifold $(M,r,\lambda(\partial M))$ with inclusion map
\begin{equation}
i_{\Sigma,M}:\Sigma\rightarrow M.
\notag
\end{equation}
Let $\Sigma'$ be $\partial M-\Sigma$ with inclusion map
\begin{equation}
i_{\Sigma',M}:\Sigma'\rightarrow M.
\notag
\end{equation} 
Then we define
\begin{equation}
\lambda_M(\Sigma)=i^{-1}_{\Sigma,M}(i_{\Sigma',M}(\lambda (\Sigma'))).
\notag
\end{equation}
\end{de}

 We define the composition of morphisms in 
 $\mathcal{C}$ as the extended gluing of cobordisms.

\begin{de}
\label{extendedgluing}
Let $(M,r,\lambda(\partial M))$ and $(M',r',\lambda(\partial M'))$ be two extended $3$-manifolds.
Suppose $(\Sigma, \lambda(\Sigma))$ is a boundary surface of $(M,r,\lambda(\partial M))$ and 
$(\bar{\Sigma}, \lambda(\Sigma))$ is  a boundary surface of $(M',r',\lambda(\partial M'))$.
Then we can glue $(M,r,\lambda(\partial M))$ and $(M',r',\lambda(\partial M'))$ together with the orientation reversing identity from $\Sigma$ to $\bar{\Sigma}$ to form a new extended $3$-manifold. The new extended $3$-manifold has
\begin{enumerate}
\item base manifold: $M\cup_{\Sigma}M'$

\item lagrangian subspace:
\begin{equation}
[\lambda(\partial M)\cap H_1(\partial M-\Sigma)]\oplus[\lambda(\partial M')\cap H_1(\partial M'-\bar{\Sigma})],
\notag
\end{equation}

\item weight:
\begin{equation}
r+r'-\mu(\lambda_M(\Sigma),\lambda(\Sigma),\lambda_{M'}(\bar{\Sigma})),
\notag
\end{equation}
where $\mu$ is the Maslov index as in \cite{T}.
\end{enumerate}
\end{de}

\begin{de}
\label{selfgluing}
Let $(M,r,\lambda(\partial M))$ be an extended $3$-manifold with a boundary surface of the form $\Sigma \cup \bar{\Sigma}$.
Then we define the extended $3$-manifold obtained by gluing $\Sigma$ and $\bar{\Sigma}$ together to be the extended $3$-manifold that results from gluing $(M,r,\lambda(\partial M))$ and $(\Sigma\times[0,1],0,\lambda (\Sigma\cup\bar{\Sigma}))$ along $\Sigma\cup\bar{\Sigma}$.
In the special case that $\partial M=\Sigma\cup\bar{\Sigma}$,  we call the resulting extended $3$-manifold the closure of $(M,r,\lambda(\partial M))$.
\end{de}

\begin{rem}
One should think of the weight of an extended $3$-manifold $M$ as the signature of some background $4$-manifold \cite{Walker}.
See also \cite[p. 399]{G4}. 
\end{rem}

\begin{lem}
\label{closure}
Let $(R,r,\lambda(\partial R))$ be a morphism from $(\Sigma,\lambda(\Sigma))$ to $(\Sigma',\lambda(\Sigma'))$ and $(S,s,\lambda(\partial S))$ be a morphism from $(\Sigma',\lambda(\Sigma'))$ to $(\Sigma, \lambda(\Sigma))$.
Then the extended $3$-manifold we obtain by gluing $(R,r,\lambda(\partial R))$ to $(S,s,\lambda(\partial S))$ along $\Sigma'$ first and then closing it up along $\Sigma$ is the same as the one we obtained from gluing $(S,s,\lambda(\partial S))$ to $(R,r,\lambda(\partial R))$ along $\Sigma$ first and then closing it up along $\Sigma'$.
\end{lem}
\begin{proof}
This can be seen from $4$-manifold interpretation of weights in \cite{Walker,GM}.\end{proof}

Extended surfaces may also be equipped with banded points: this is an embedding of the disjoint union of oriented intervals. By a framed link, we will mean what is called a banded link in  \cite[p.884]{BHMV}, i.e. an embedding of the disjoint union of oriented annuli.  Framed $1$-manifold are defined similarly. Extended $3$-manifolds are sometimes equipped with framed links, or framed 1-manifolds  or more generally trivalent fat graphs.  
 By a trivalent fat graphs, we will mean what is called a banded graph in  \cite[p.906]{BHMV}. The framed links, framed 1-manifolds and trivalent fat graphs must meet the boundary surfaces of a 3-manifold in banded points with the induced ``banding''.  Of course,  we could have used the word ``banded'' in all cases, but the other terminology is more common.

There is  a surgery theory for extended $3$-manifolds.
We refer the reader to \cite[\S 2]{GM}.
Here we give extended version of Kirby moves \cite{K}. 
 These moves relate  framed links in  $S^3$ where $S^3$ is itself equipped with an integer weight. The result of extended surgery of $S^3$ with its given weight along the link is preserved by these moves. Moreover (but we do not use this) if surgery along two framed links in weighted copies of $S^3$ result in the same extended manifold then there is sequence of extended Kirby moves relating them.

\begin{de}
The extended Kirby-$1$ move is the regular Kirby-$1$ move with weight of manifold changed accordingly. 
More specifically, if we add an $\epsilon$-framed unknot  to the surgery link, then we change the weight of the manifold by $-\epsilon$, where $\epsilon=\pm 1$. If we delete an  $\epsilon$-framed unknot from the surgery link, then we change the weight of the manifold by $\epsilon$.
The extended Kirby-$2$ move is the regular Kirby-$2$ move with the weight remaining the same.
\end{de}

\subsection{A variant of the TQFT of Blanchet, Habegger, Masbaum and Vogel.}
Suppose a closed connected $3$-manifold $M$ is obtained from $S^3$ by doing surgery along a framed link $L$, then
 $(M,r)$ is obtained from $(S^3,r-\sigma(L))$ by doing extended surgery along $L$.
Here $\sigma(L)$ is the signature of  the linking matrix of $L$.  
 Warning this is different than the signature of $L$ .
The  quantum invariant of $(M,r)$ at a $2p$th root of unity $A$ is then defined as:
\begin{equation}
Z((M,r))=\eta\kappa^{r-\sigma(L)}<L(\omega)>,
\notag
\end{equation}
where
\begin{equation}
\Delta_k= <U(e_k)>, \hskip 0.1in \eta^{-1}=\sqrt{\sum_k\Delta_k^2},\hskip 0.1in\omega=\sum_k\eta\Delta_k e_k,\hskip 0.1in\kappa=<U_+(\omega)>.
\notag
\end{equation}
We use  $< \ >$  to denote the Kauffman bracket evaluation of a linear combination of colored links in $S^3$, and $\mathcal{L}(x)$ to denote the satellization of a framed link $\mathcal{L}$ by a skein $x$ of the solid torus.  Moreover $e_k$ denotes the skein class  in the solid torus obtained by taking the closure of $f_k$,   the   Jones-Wenzl idempotent in the $k$-strand Temperley-Lieb algebra. Here $U$ denotes the zero framed unknot and  $U_+$ is the unknot with framing $+1$. 
The sum is over the colors $0\leq k\leq p/2-2$ 
 if $p$ is even and $0\leq k\leq p-3$ with $k$ even if $p$ is odd.
One has that $\kappa$ is a square root of $A^{-6-p(p+1)/2}$. 
The choice of square root here determines the choice in  the square root in the formula of $\eta^{-1}$, or vice-versa. 
See the formula for $\eta$ in \cite[page 897]{BHMV}.
The  closed connected  manifold $M$ may also have an embedded $p$-admissibly colored
fat trivalent graph $G$ in the complement of the surgery, then
$$Z((M,r),G)=\eta\kappa^{r-\sigma(L)}<L(\omega) \cup G>.$$
By following the exactly the same procedure in \cite{BHMV}, 
we can construct a TQFT for the category of extended surfaces and extended $3$-manifolds from quantum invariants. 
The TQFT assigns to each extended surface $(\Sigma, \lambda(\Sigma))$, possibly with some
 banded colored points, 
a module $V(\Sigma,\lambda(\Sigma))$ over $k_p=\bz[{\frac 1
 p},A, \kappa]$, and assigns  to each extended cobordism $M$, 
 with a $p$-admissibly colored trivalent
  fat graph meeting the banded colored points, 
\begin{equation}
(M,r,\lambda(M)):(\Sigma,\lambda ({\Sigma}))\rightarrow(\Sigma',\lambda (\Sigma'))
\notag
\end{equation}
a $k_p$-module homomorphism:
\begin{equation}
Z((M,r,\lambda(M))):V((\Sigma,\lambda ({\Sigma}))\rightarrow V((\Sigma',\lambda (\Sigma'))).
\notag
\end{equation}
Then by using this TQFT, we can produce a Turaev-Viro endomorphism associated to  each weighted 
closed 3-manifold equipped with a choice of infinite cyclic cover using the procedure described in
 \S \ref{introduction}.

\begin{note}
We introduce some notations that will be used later.
\begin{enumerate}
\item $\Lambda^{(l)}_k=\eta^l\Delta^l_kf_k$,
\item $\omega^{(l)}=\sum_k\eta^l\Delta^l_k e_k$,  
\item $\Theta(a,b,c)$ is the Kauffman bracket of the left diagram in Figure \ref{thetatet},
\item $Tet(a,b,c,d,e,f)$ is the Kauffman bracket of the right diagram in Figure \ref{thetatet}.
\end{enumerate}

\begin{figure}[h]
\includegraphics[width=2in,height=1in]{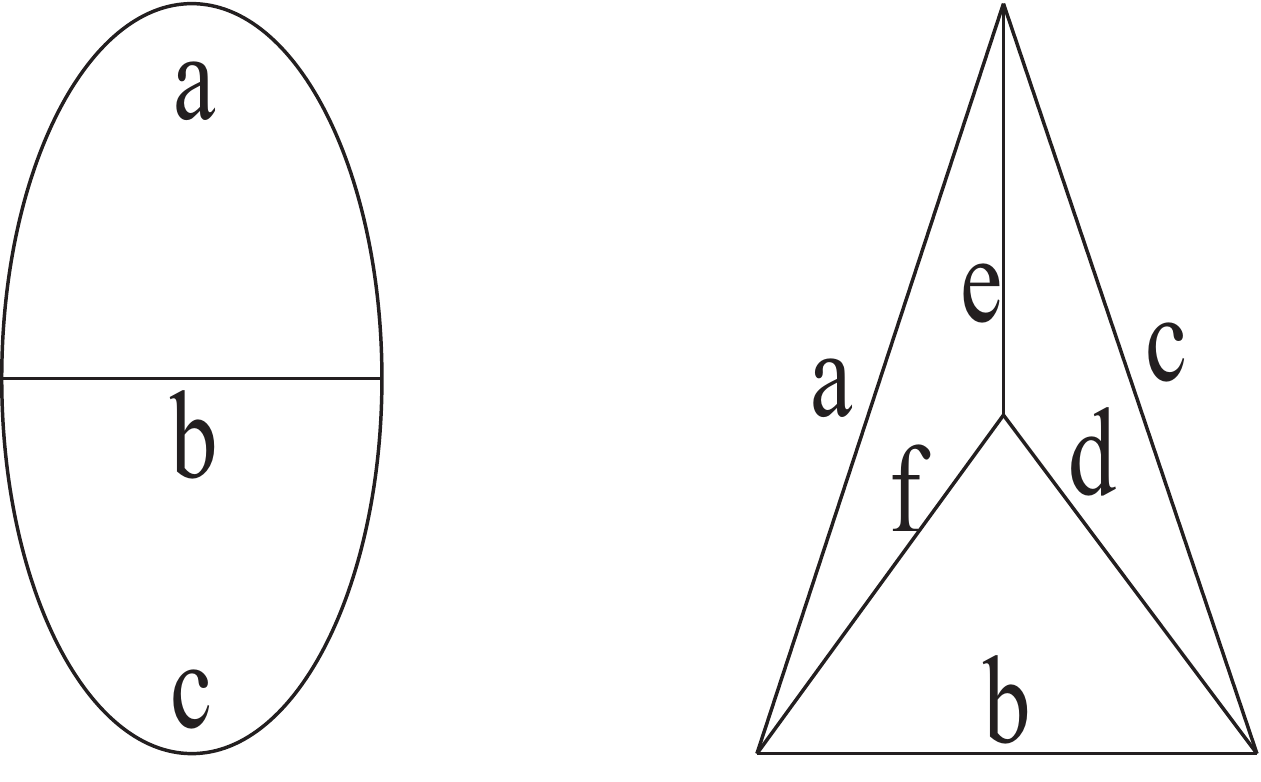}
\caption{On the left is $\Theta(a,b,c)$, and on the right is $\Tet(a,b,c,d,e,f)$.}
\label{thetatet}
\end{figure}
\end{note}

\subsection{Surgery presentations}

The earliest use of surgery presentations, that we are aware of, was  by Rolfsen \cite{R2} to compute and study the Alexander polynomial. 
In this paper we consider surgery descriptions for extended closed 3-manifolds with an infinite cyclic cover. 
We will  use these descriptions for extended 3-manifolds that contain certain colored trivalent fat graphs. 
As this involves no added difficulty, we will not always mention these graphs in this discussion.

\begin{de}
Let $K_0\cup L$ be a framed link inside $(S^3,s)$ where $K_0$ is an oriented $0$-framed unknot, 
and the linking numbers of the components of $L$ with $K_0$ are all zero. 
Let $D_0$ be a disk in $S^3$  with boundary $K_0$ which is transverse to $L$.
Suppose $(M,r)$ is the result of extended surgery along $K_0\cup L$,
then there exists a unique epimorphism $\chi: H_1(M,\bz)\rightarrow\bz$ which agrees with the linking number with $K_0$ on cycles in $S^3\setminus(K_0\cup L)$. 
We will call  $(D_0,L,s)$ a surgery presentation of $((M,r),\chi)$. 
We  remark that, in this situation, we will have $s= r -\sigma(L)$. 
If there are  graphs $G'$ in $M$ and $G$ in $S^3\setminus (K_0\cup L)$ (transverse to $D_0$)   related by the surgery, 
we will say $(D_0,L,s,G)$ a surgery presentation of $((M,r),\chi,G')$.
\end{de}

If the result of  surgery along $L$ returns $S^3$ with the image of $K_0$ after surgery becoming a knot  oriented knot $K$,  and the linking numbers of the components of $L$       
with $K_0$, then $K_0\cup L$ is a surgery presentation of $K$ as in Rolfsen. 
The manifold obtained by surgery along $K_0\cup L$ in $S^3$ is the same as $0$-framed surgery along $K$ in $S^3$.

The following Proposition  is proved  in section 4 of \cite{O} for non-extended manifolds.  The extended version involves no extra difficulty
 
\begin{prop}
Every extended   connected $3$-manifold with an epimorphism $\chi: H_1(M,\bz) \rightarrow \bz$  has a surgery presentation.
\end{prop}

Every surgery presentation can be described  by  diagram as in  Figure \ref{almoststandardform} which we will refer
to as a surgery presentation diagram.

\begin{figure}[h]
\includegraphics[width=1.2in,height=1.2in]{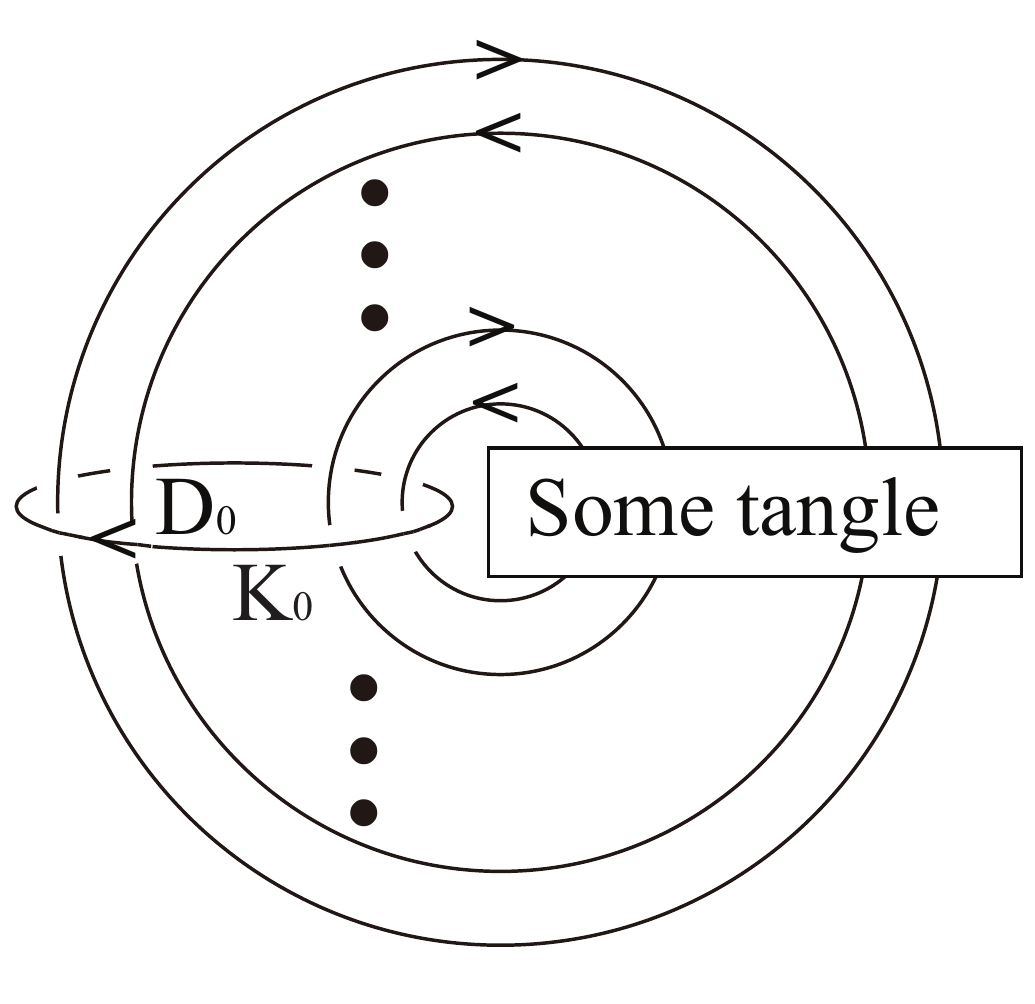}
\caption{A surgery presentation diagram.  Of course, the tangle must be such that each closed component of $L$ has zero linking number with $K_0$. Notice the orientation on $K_0$. }
\label{almoststandardform}
\end{figure}

\begin{de}
If a surgery presentation diagram  is in the form of Figure \ref{standardform},
then we say this surgery presentation diagram  is in standard form. We will also say that  a surgery presentation $(D_0,L,s,G)$ is standard  if it has a surgery presentation diagram  in standard form.
\end{de}
\begin{figure}[h]
\includegraphics[width=1.2in,height=1.2in]{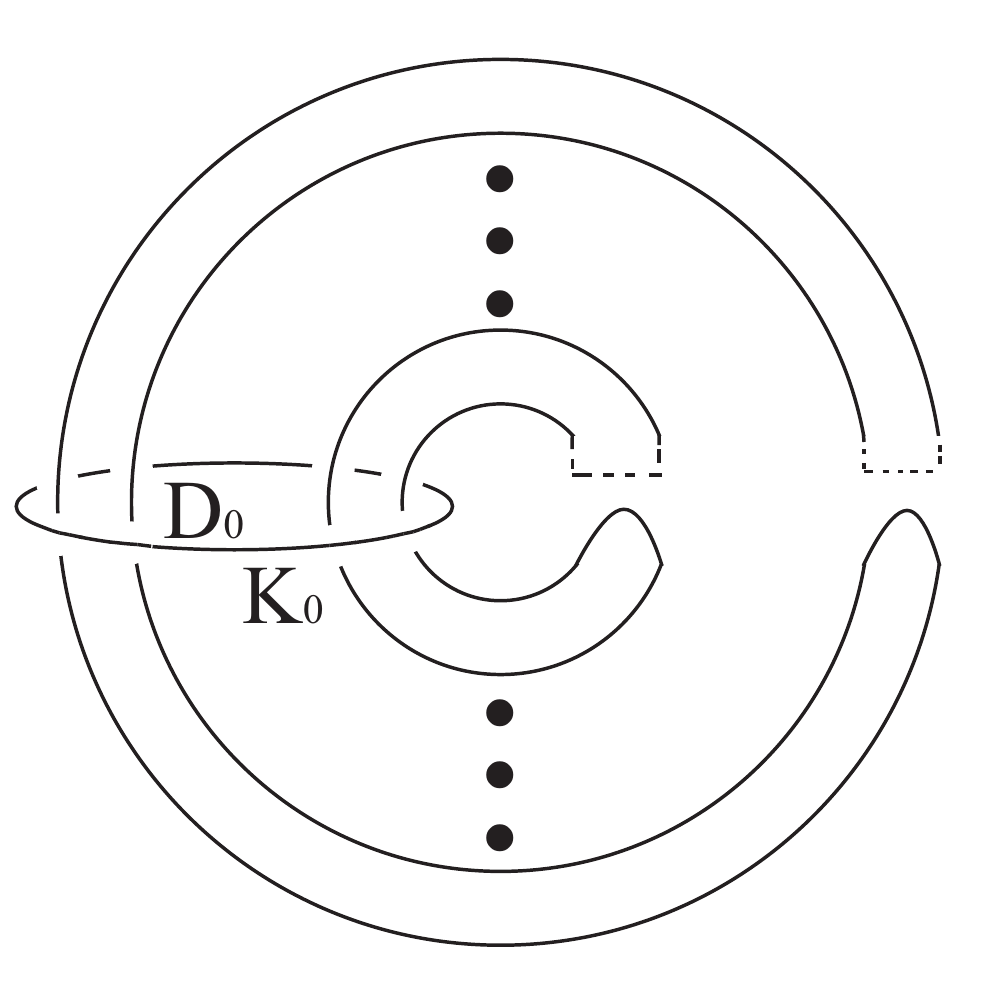}\hskip 0.5in
\caption{The dotted part could be knotted or linked with other strands within the tangle box.
The bottom turn-backs are simple arcs without double points under the projection.
Each component of $L$ intersects the flat disc $D_0$ bounded by the trivial knot algebraically $0$ times,
but geometrically $2$ times.}
\label{standardform}
\end{figure}

 Ohtsuki \cite[bottom of p. 259]{O}  stated a  proposition about  surgery presentations of knots 
 which is similar to the following proposition.  Our proof is similar to the proof that Ohtsuki indicated. 
We will call a Kirby-$1$ move in a surgery presentation a small Kirby-$1$ move if a disk which bounds  the created or deleted  component is  in the complement of $D_0$.
We will call a Kirby-$2$ move in a surgery presentation a small Kirby-$2$ move if it involves sliding a component other than $K_0$ over another component that is  in the complement of $D_0$.
A $D_0$-move is a choice of a new spanning disk $D'$  with $D_0 \cap D'=K_0$ followed by an ambient isotopy that moves $D'$ to the original position of $D_0$ and moves $L$ at the same time. 

\begin{prop}
\label{standardization}
A surgery presentation described by a surgery presentation diagram  can be transformed into a surgery presentation described by a surgery presentation diagram  in  standard form by a sequence of isotopies of  $L \cup G$ relative to $D_0$,  small Kirby-$1$ moves, small Kirby-$2$ moves, and $D_0$-moves.
Therefore, every extended $3$-manifold with an epimorphism $\chi: H_1(M,\bz)\rightarrow\bz$ has a standard surgery presentation.
\end{prop}

\begin{proof}
We need to prove that we can change a surgery presentation described by a surgery diagram as in Figure \ref{almoststandardform} into surgery presentation described  by a diagram as in Figure \ref{standardform}
using the permitted moves.

Let
\begin{equation}
m=\max_{L_i \text{\ is a component of\ } L}|L_i\cap D_0|.
\notag
\end{equation}
We will prove the theorem by induction on $m$.
Since each component $L_i$ has linking number 0 with $K_0$,
it is easy to see that $m$ is even.

If $m=0$, then $L$ can be taken to be contained in the tangle box. 

When $m=2$, we may 
\begin{itemize}
  \item first  do a $D_0$ move to shift  $D_0$ slightly;  
 \item then perform an isotopy relative to  the new $D_0$ of $L$ so that  the points on intersection of  the image of the old $D_0$ with each components of $L$  are adjacent to each other;
 \item  then do another $D_0$-move to move the old $D_0$ back to its original position.
 \end{itemize}

Now the arcs  emitted from the bottom edge of the tangle are in a correct  order. 
But the diagram in Figure \ref{almoststandardform} may differ from a standard tangle in the way that the arcs emitted from bottom edge of the tangle box are not in the specified simple form. This means they could be knotted and linked with each other. However we may 
perform  small Kirby-$1$ and small Kirby-$2$ moves as in Figure \ref{changecrossing}  to unknot and unlink these arcs so that the resulting diagram has  standard form.

\begin{figure}[h]
\includegraphics[width=0.6in,height=1in]{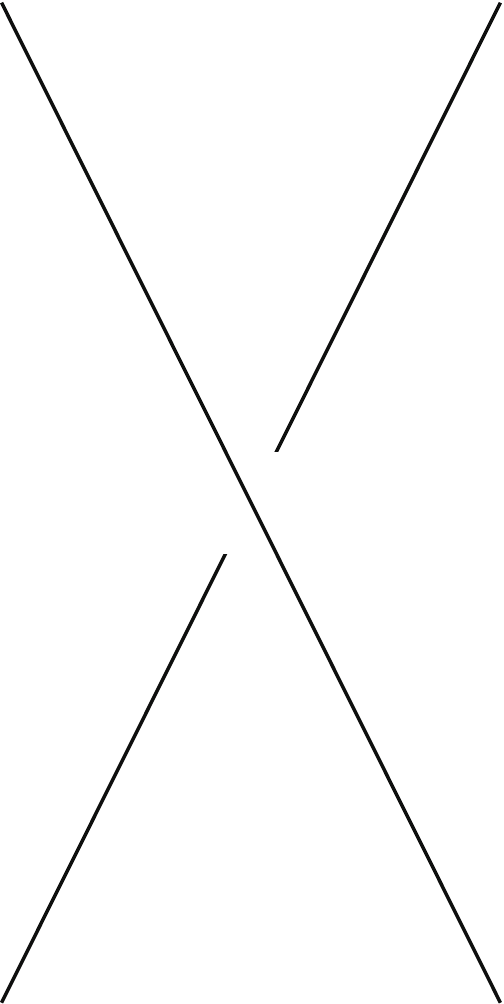}
\includegraphics[width=0.3in,height=1in]{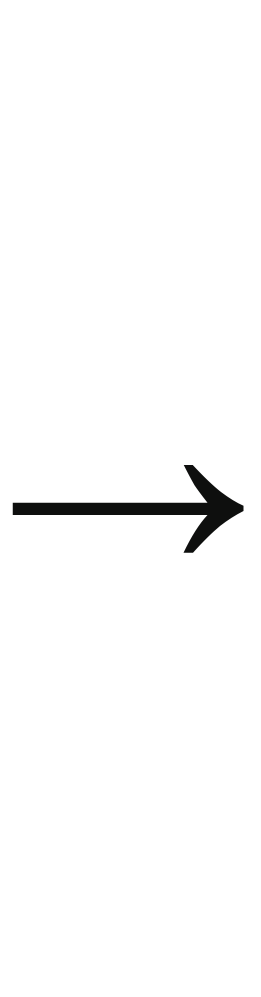}
\includegraphics[width=0.8in,height=1in]{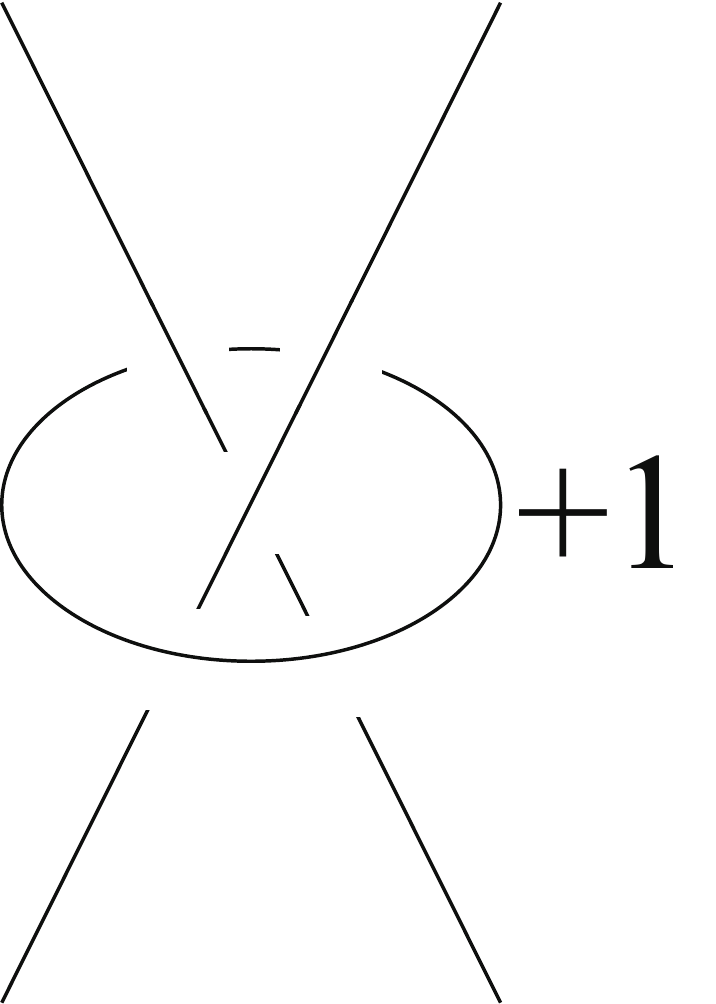}\hskip 0.2in
\includegraphics[width=0.6in,height=1in]{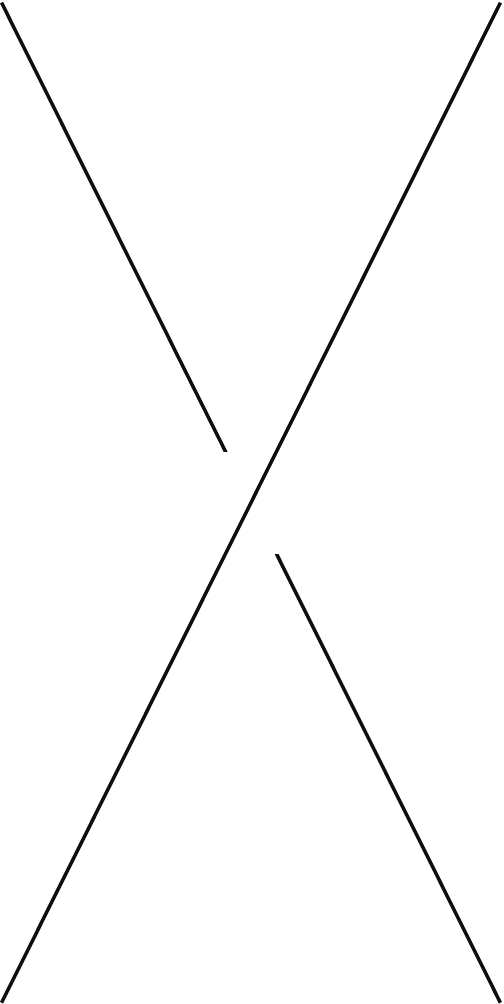}
\includegraphics[width=0.3in,height=1in]{harrow.pdf}
\includegraphics[width=0.8in,height=1in]{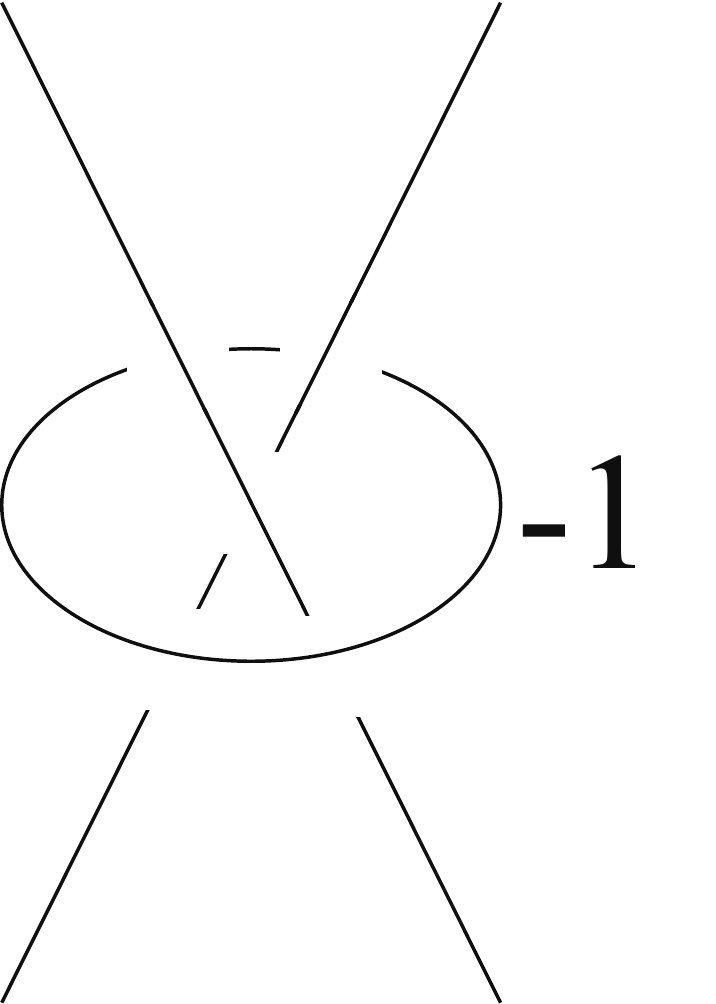}
\caption{We use $+1$ or $-1$ surgery on unknot to change the crossing.}
\label{changecrossing}
\end{figure}

 We now  prove that the theorem holds for all links with $m=2n$ where  $n \ge 2,$  assuming it holds for  all links with $m \le 2n-2$.
Suppose the component $L_1$ intersects $D_0$ geometrically $2n$ times.
Because $L_1$ has linking number $0$ with $K_0$,
we have that at least one  arc, say $\alpha$  of $L_1$ in Figure \ref{almoststandardform} which joins two points on the bottom of the  tangle box,
i.e. it is a ``turn-back''.
For each crossing with exactly one arc from $\alpha$,
we can make the arc  $\alpha$ to be the top arc (in the direction perpendicular to the plane of  the diagram) by using the moves of Figure \ref{changecrossing},
which just involve some small Kirby-$1$ and small Kirby-$2$ moves.
Then it is only simply linked to other components by some new trivial components with framing $\pm1$.
Then by using isotopies relative  to $D_0$, we can slide  the arc $\alpha$ towards bottom of the tangle, with the newly created unknots stretched vertically in the diagram so that they intersect each horizontal  cross-section in at most 2-points.   
See the central illustration  Figure \ref{example2} where $\alpha$ is illustrated by two vertical arcs meeting a small box labeled $X$. This small box contains the rest of $\alpha$.  
Now perform a $D_0$ move  which has the  effect of  pulling the turn-back across $D_0$.
Those trivial components will follow the turn back and pass through $D_0$.
But since at the beginning, those components have geometric intersection $0$ with $D_0$,
they have geometric intersection $2$ with $D_0$ now.
After this process, $L_1\cap D_0$ is reduced by $2$.
This process does not change the number of  intersections with $D_0$ of the other components of the original $L$.

\begin{figure}[h]
\includegraphics[width=0.8in,height=1in]{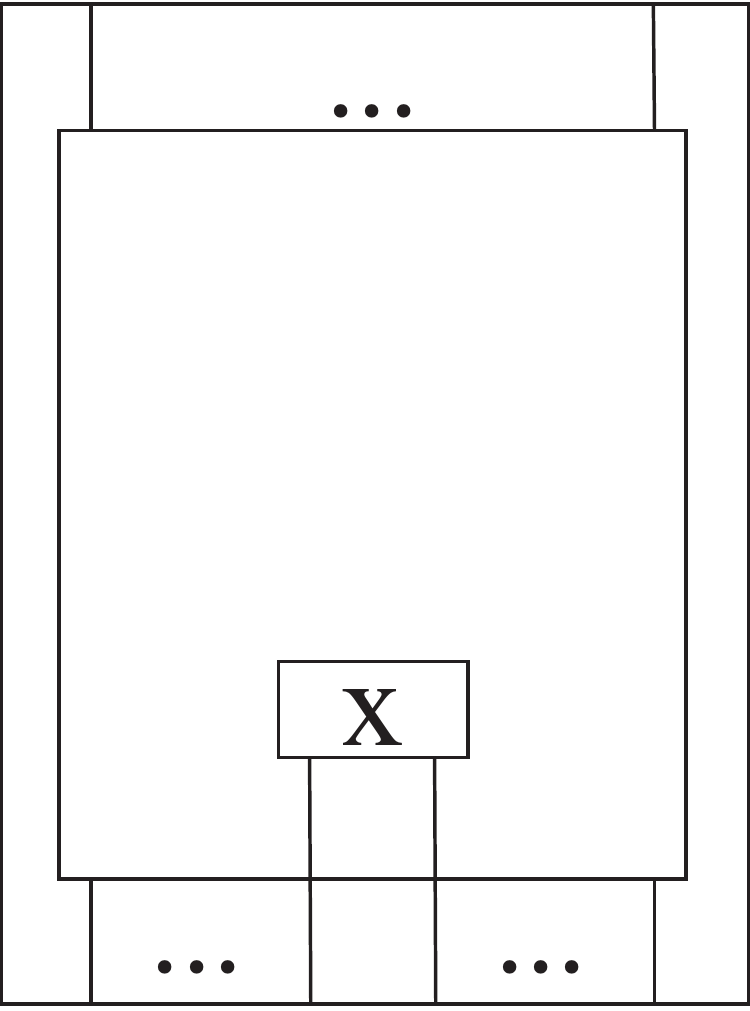}
\includegraphics[width=0.3in,height=1in]{harrow.pdf}
\includegraphics[width=0.8in,height=1in]{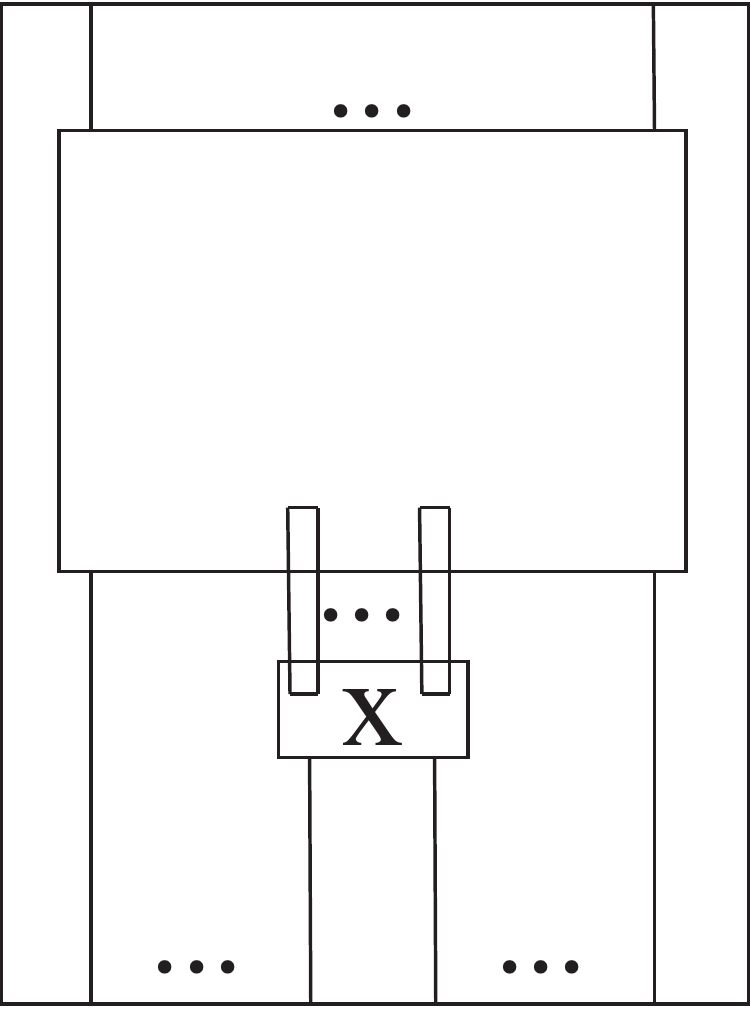}
\includegraphics[width=0.3in,height=1in]{harrow.pdf}
\includegraphics[width=0.8in,height=1in]{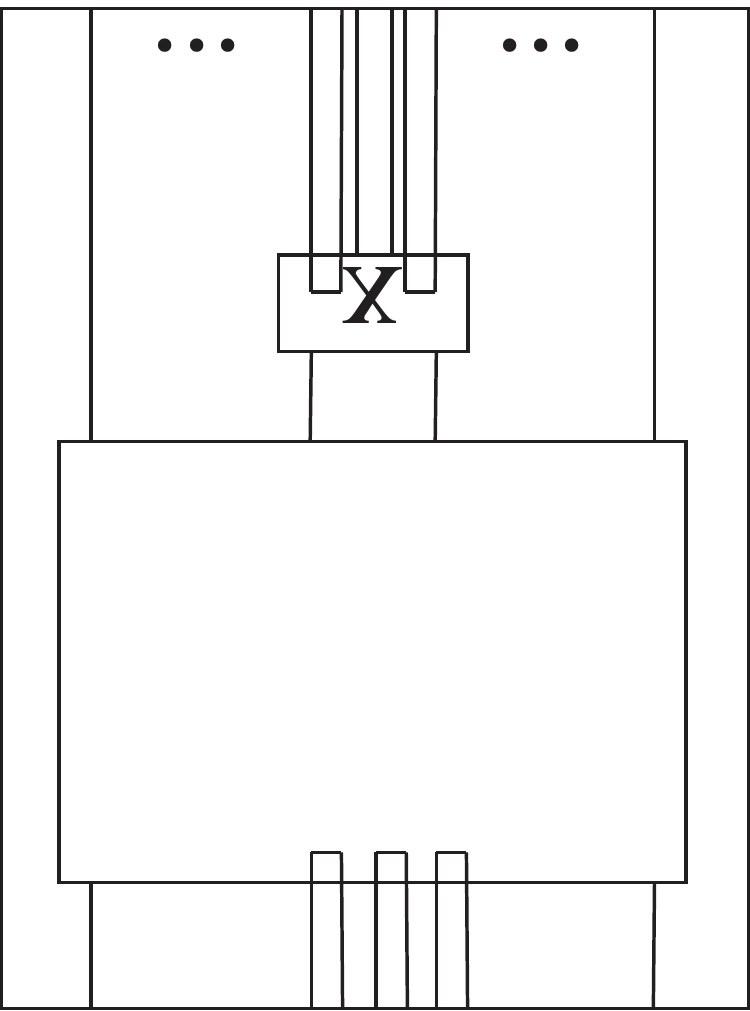}
\caption{Moves which reduce the number of intersections of a component of $L$  with $D_0$. We perform small K-moves and isotopies to change to the middle picture. 
We perform a $D_0$-move to change to the right hand picture. }
\label{example2}
\end{figure}

We do this process for all components $L_j$ with $|L_j\cap D_0|=2n$.
Then the new link has $m \le 2n-2$.
By our induction hypothesis, we can transform $K_0\cup L$ into a standard form using the allowed moves.
\end{proof}

\subsection{Strong shift equivalence. \label{sse}} 

We will discuss SSE in the category of free finitely generated modules over a commutative ring 
with identity. 
This notion arose in symbolic dynamics. For more information, 
see \cite{wagoner,LM} and references therein.

\begin{de}
Suppose
\begin{equation}
X:V\rightarrow V, Y:U\rightarrow U
\notag
\end{equation}
are module endomorphisms.
We say $X$ is elementarily strong shift equivalent to $Y$ if there are two module morphisms
\begin{equation}
R:V\rightarrow U, S:U\rightarrow V
\notag
\end{equation}
such that
\begin{equation}
X=SR,Y=RS.
\notag
\end{equation}
We denote this by $X\approx Y$.
\end{de}

\begin{de}
Suppose
\begin{equation}
X:V\rightarrow V,Y:U\rightarrow U
\notag
\end{equation}
are module endomorphisms.
We say $X$ is strong shift equivalent to $Y$ if there are finite number of module endomorphisms $\{X_1,...,X_n\}$
such that
\begin{equation}
X\approx X_1\approx X_2...\approx X_n\approx Y.
\notag
\end{equation}
We denote this by $X\sim Y$.
\end{de}

It is easy to see that if $X\sim Y$,  then  $\Trace (X)= \Trace (Y)$.

\begin{prop}\label{ker}  Let $X$ be a module endomorphism of  $V$.  Suppose $V= U\oplus W$ where  $U$ and $W$  are free finitely generated modules such that  $U$ in the kernel of $X$, and let  $\hat X$ be the induced endomorphism of $W$, then $\hat X$ is SSE to $X.$
\end{prop}

\begin{proof}
Suppose $\rank(U)=m$, and  $\rank(W)=n$. The result follows from the following block matrix
equations. 
\begin{equation*}\label{sse1}
\left[
\begin{array}{c}
v_{m \times n}
 \\ 
\hline
\hat X_{n \times n}
\end{array}\right]_{(n+m) \times n} 
  \cdot 
\left[
\begin{array}{c|c}
  0_{n\times m}
 &
  I_n
\end{array}
 \right]_{n \times (n+m)}
=
\left[
\begin{array}{c|c}
0_{m \times m} & v_{m \times n} \\ \hline
0_{n \times m} &  \hat X_{n \times n}
\end{array}\right]_{(n+m) \times (n+m)} 
\end{equation*}

\begin{equation*}\label{sse2}
\left[
\begin{array}{c|c}
  0_{n\times m}
&
I_n
 \end{array}
 \right]_{n \times (n+m)}
\cdot 
\left[
\begin{array}{c}
v_{m \times n}
 \\ 
\hline
\hat X_{n \times n}
\end{array}\right]_{(n+m) \times n} 
=
\left[\hat X_{n \times n} 
\right]_{n \times n} 
\end{equation*}

\end{proof}

If $T$ is an endomorphism of a vector space $V$, let $N(T)$ denote the generalized $0$-eigenspace for $T$, and 
$T_\flat$ denote the induced endomorphism on $V/N(T)$.   The next proposition may be deduced 
from more general statements made in \cite[p. 122, Prop(2.4) ]{BH}. For the convenience of the reader, we give  direct proof.

\begin{prop}\label{similar} Let $T$ and $T'$ be  endomorphisms of  vector spaces. $T$ and $T'$ are SSE if and only if  
$T_\flat$ and $T_\flat'$ are similar.
\end{prop}

\begin{proof}  The only if implication   is well-known \cite[Theorem 7.4.6]{LM}.   The if implication follows from the  easy observations that similar  transformations are strong shift equivalent and that $T$ is strong shift equivalent to $T_\flat$.
This second fact follows from the repeated use of the following observation: If  $x\ne 0$ is in the null space of $T$,   $<x>$ denotes the space spanned by $x$, and $T_x$ denotes  the induced map on $V/<x>$, then $T$ and $T_x$ are  strong shift equivalent. This follows from Proposition \ref{ker} with $U= <x>$.
\end{proof}


\section{The tangle morphism}
\label{thetanglemorphism}
In this section, 
we will assign a $k_p$-module homomorphism to any framed tangle in $S^2\times I$  
enhanced with an embedded $p$-admissibly colored trivalent fat graph in the complement of the tangle. By slicing a surgery presentation for an infinite cyclic cover of an extended $3$-manifold and applying the TQFT, we obtain such a tangle, and thus a 
$k_p$-module endomorphism.
The idea of constructing this endomorphism is inspired by the work of Ohtsuki in \cite{O}.

There is a unique lagrangian for a 2-sphere.  Thus we can consider any 2-sphere as an extended manifold without specifying a lagrangian.
Similarly,  we let  $(S^2\times I,r)$ denote the extended manifold $S^2 \times I$ with weight $r$, as there is no need to specify a lagrangian.

\begin{de}
Let $S$ be a 2-sphere  equipped with $m$ ordered uncolored banded points, and $u$ ordered banded points colored by $x_1,\cdots,x_u$.
We define $S(i_1, i_2, \cdots i_m)$ to be this 2-sphere where the  $m$ uncolored banded points
have been colored by \linebreak $(i_1, i_2, \cdots i_m)$ (and the $u$ points already colored remain colored).

We define
\begin{equation}
V(S)=\sum_{i_1,...,i_m} V(S(i_1, i_2, \cdots i_m)).
\notag
\end{equation}
Here $V(S(i_1, i_2, \cdots i_m))$ is the module for a extended 2-sphere with  $m$ uncolored  banded points colored by $(i_1,\cdots,i_m)$ and $u$ banded points colored by $(x_1,\cdots,x_u)$ obtained by applying the TQFT that we introduced in \S  \ref{preliminary}.
\end{de}

By an $(m,n)$-tangle in $(S^2\times I,r)$, we mean  
  a properly embedded framed 1-manifold   in $(S^2\times I,r)$ with $m$ endpoints on $S_0=\sto$, $n$ points on $S_1=\stl$, with possibly some black dots on its components and a (possibly empty) colored trivalent fat graph (in the complement of the 1-manifold)  meeting $S_0$ in  $u$ colored points  $x_1, \cdots, x_u$   and meeting  $S_1$ in  $t$ colored points 
  $y_1,\cdots,y_t$. Thus $S_0$ is a 2-sphere with $m$ ordered uncolored  banded points and $u$ colored banded points. Similarly $S_1$ is a 2-sphere with $n$ ordered uncolored banded points and $t$ colored banded points. 
For any $(m,n)$-tangle, we will define a homomorphism from $\vsm$ to $\vsn$. 

Before doing that, we introduce some definitions.  From now on, we will not explicitly mention the banding on the selected points of a surface
or the framing of a tangle, or the fattening of a trivalent graph. Each comes equipped with
such and the framing/fattening of a link/graph induces the banding on its boundary points. Nor will we mention the ordering chosen for uncolored  sets of points.

\begin{de}
\label{tangle}
Suppose we have a $(m,n)$-tangle in $(S^2\times I,r)$ with a colored trivalent graph  with $u$ edges colored by $x_1, \cdots, x_u$ 
meeting $\sto$ and $t$ edges colored by $y_1,\cdots,y_t$ meeting $\stl$.
Suppose we color the $m$ endpoints from the tangle on $S_0=\sto$ by $i_1,\cdots,i_m$ and color the $n$ endpoints from the tangle on $S_1=\stl$ by $j_1,\cdots,j_n$.
We say that the coloring $(i_1,\cdots,i_n,j_1,\cdots,j_m)$ is legal if the two endpoints of the same strand have the same coloring.
We denote the tangle with the endpoints so-colored by $T^{n,(j_1,...,j_n)}_{m,(i_1,...,i_m)}$.
For an example, see Figure \ref{colorings}.
\end{de}
\begin{figure}[h]
\includegraphics[width=1in,height=1.25in]{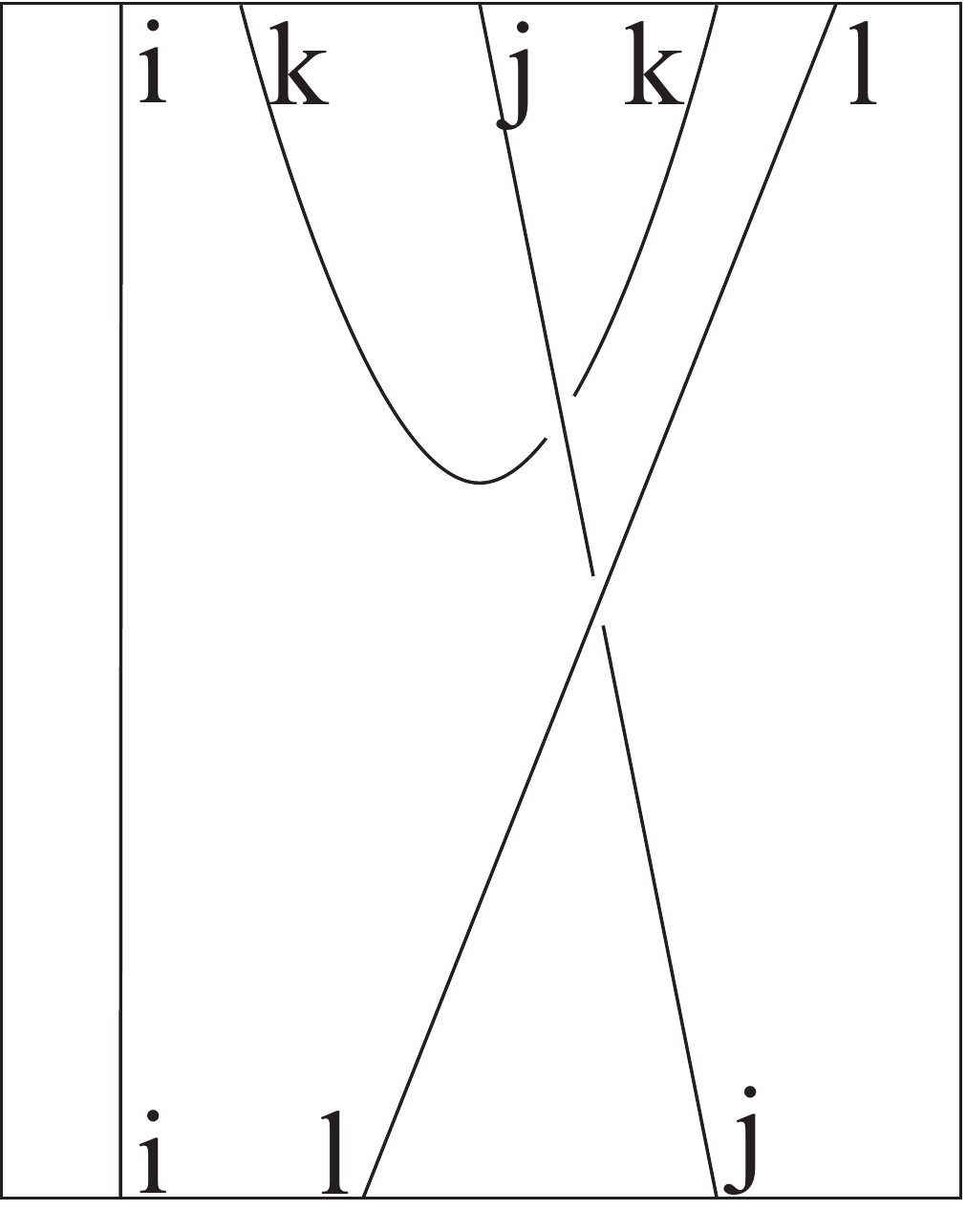}\hskip 0.5in
\includegraphics[width=1in,height=1.25in]{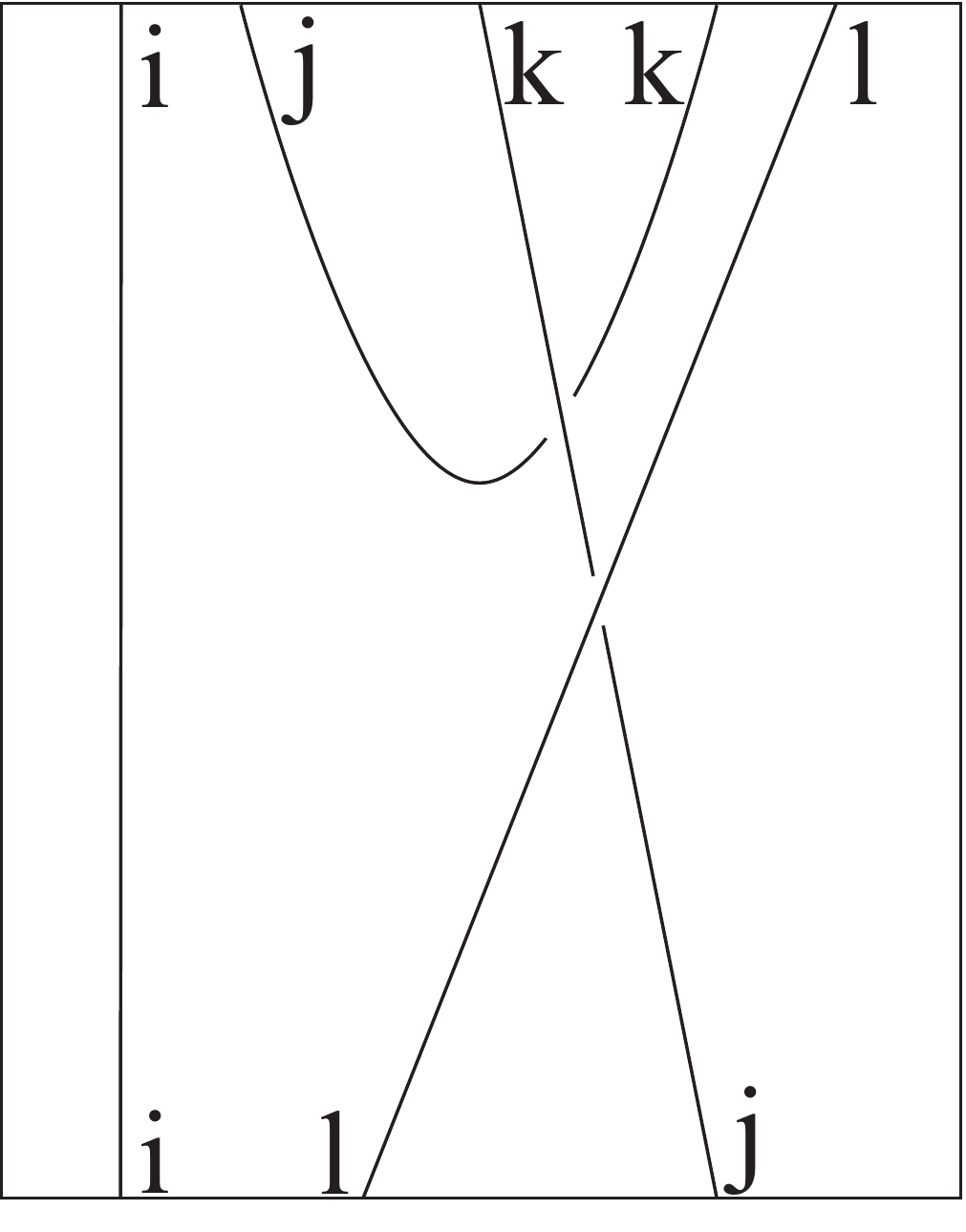}
\caption{The coloring in first diagram is a legal coloring and the one in the second diagram is a illegal coloring for $k\neq j$.
In this example, the colored trivalent graph is empty.}
\label{colorings}
\end{figure}

\begin{de}
\label{componenttanglemap}
Suppose we have $T^{n,(j_1,...,j_n)}_{m,(i_1,...,i_m)}$, a $(i_1,\cdots,i_m,j_1,\cdots,j_n)$ colored $(m,n)$-tangle as in Definition \ref{tangle}.
We define a homomorphism
\begin{equation}
\vsim\stackrel{\tnmij}{\longrightarrow}\vsjn
\notag
\end{equation}
as follows.
\begin{itemize}
\item If $(i_1,\cdots,i_m,j_1,\cdots,j_n)$ is a illegal coloring. 
We take the homomorphism to be the zero homomorphism.
\item If $(i_1,\cdots,i_m,j_1,\cdots,j_n)$ is a legal coloring.
We decorate uncolored components of the tangle by some skeins in $2$ cases:
\begin{enumerate}
\item If there are $l$ black dots on the component, $l\in\{0,1,2,\cdots\}$, 
and the component has two endpoints with color $k$, $k\in\{i_1,\cdots,i_m,j_1,\cdots,j_n\}$,
then we decorate the component by $\Lambda^{(l)}_k$.

\item If there are $l$ black dots on the component, $l\in\{0,1,2,\cdots\}$, 
and the component lies entirely in $S^2\times (0,1)$,
then we decorate the component by $\omega^{(l)}$.
\end{enumerate}
Then we apply $Z$ to $(S^2\times I,r)$ with the tangle $\tnm$, so decorated,   to get the morphism $\tnmij$.
\end{itemize}
\end{de}

Now we are ready to define the homomorphism for a tangle $\tnm$.

\begin{de}
\label{tanglemap}
Suppose we have a $(m,n)$-tangle $\tnm$.
We define the homomorphism for the tangle, denoted by $\ztnm$, to be
\begin{equation}
\vsm\stackrel{\sum\tnmij}{\longrightarrow}\vsn
\notag
\end{equation}
where $\tnmij$ is as in Definition \ref{componenttanglemap} and the sum runs over all colorings $(i_1,\cdots,i_m,j_1,\cdots,j_n)$.
\end{de}

\begin{prop}
\label{functorproposition}
For a tangle $T_1$ in $(S^2\times I,r)$ and a tangle $T_2$ in $(S^2\times I,s)$,
we have
\begin{equation}
Z(T_2\circ T_1)=Z(T_2)Z(T_1),
\notag
\end{equation}
where $T_2\circ T_1$ in $(S^2\times I,r+s)$ means gluing $T_2$ on the top of $T_1$. 
Here, of course, we assume that the top of $T_1$  and the bottom  of $T_2$ agree.  
\end{prop}
\begin{proof}
This follows from  the functoriality  of the original TQFT. 
\end{proof}

Now we can construct tangle endomorphisms for an extended closed $3$-manifold with an embedded colored trivalent graph, and choice of infinite cyclic cover.
Given $((M,r),\chi,G')$,
we choose a surgery presentation $(D_0,L,s,G)$.
We put one black dot somewhere on each component of $L$ away from $D_0$.
By doing a $0$-surgery along $K_0$, we obtain $(S^2\times S^1,s)$ with link $L$ and trivalent graph $G$, where  $D_0$ can be completed to $\stp$ for some
 point $p$ on $S^1$.  We cut $S^2\times S^1$ along $\stp$ .
Then we obtain a tangle $\mathcal{T}^n_n$ in $(S^2\times I,s)$. 
Here  $n= |\mathcal{T}^n_n\cap (\stl)| = |\mathcal{T}^n_n \cap (\sto)|$. 
 Let $Z(\mathcal{T}^n_n)$ denote tangle endomorphism associated to $\mathcal{T}^n_n$.

\begin{lem}
\label{blackdots} If $\mathcal{T}^n_n$ is constructed as above, then
the SSE class of $Z(\mathcal{T}^n_n)$ is independent of the positioning of the black dots.
\end{lem}
\begin{proof}
By definition, we can move a black dot on the component of the tangle $\mathcal{T}^n_n$ anywhere without changing the tangle endomorphism $Z(\mathcal{T}^n_n)$.
We move the black dot to near bottom or near top and cut the  tangle $\mathcal{T}^n_n$ into two tangles $S$ and $T$,
where $T$ is a trivial tangle with the black dot.
For an example, see Figure \ref{blackdot}.
Then we switch the position of $S$ and $T$ and move the black dot in resulting tangle to near the other end of that component.
Then we do the process again.
By doing this, we can move it to any arc of the tangle $\mathcal{T}^n_n$, 
which belongs to the same component of the link $L$.
But for each step, $Z(ST)=Z(S)Z(T)$ is strong shift equivalent to $Z(TS)=Z(T)Z(S)$.
Therefore, the lemma is true.
\end{proof}
\begin{figure}[h]
\includegraphics[width=1.2in, height=1.2in]{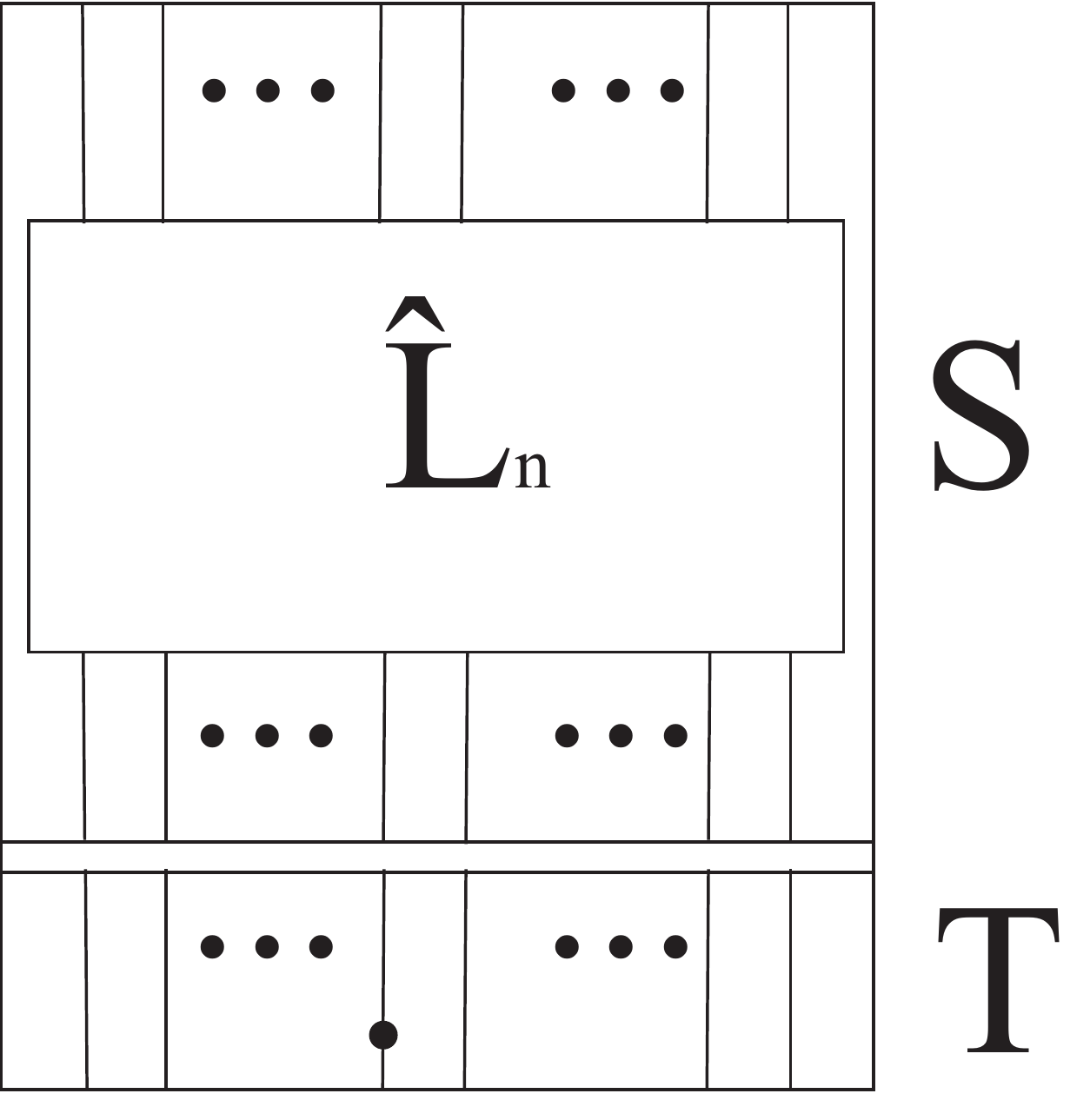}
\includegraphics[width=0.3in, height=1.2in]{harrow.pdf}
\includegraphics[width=1.2in, height=1.2in]{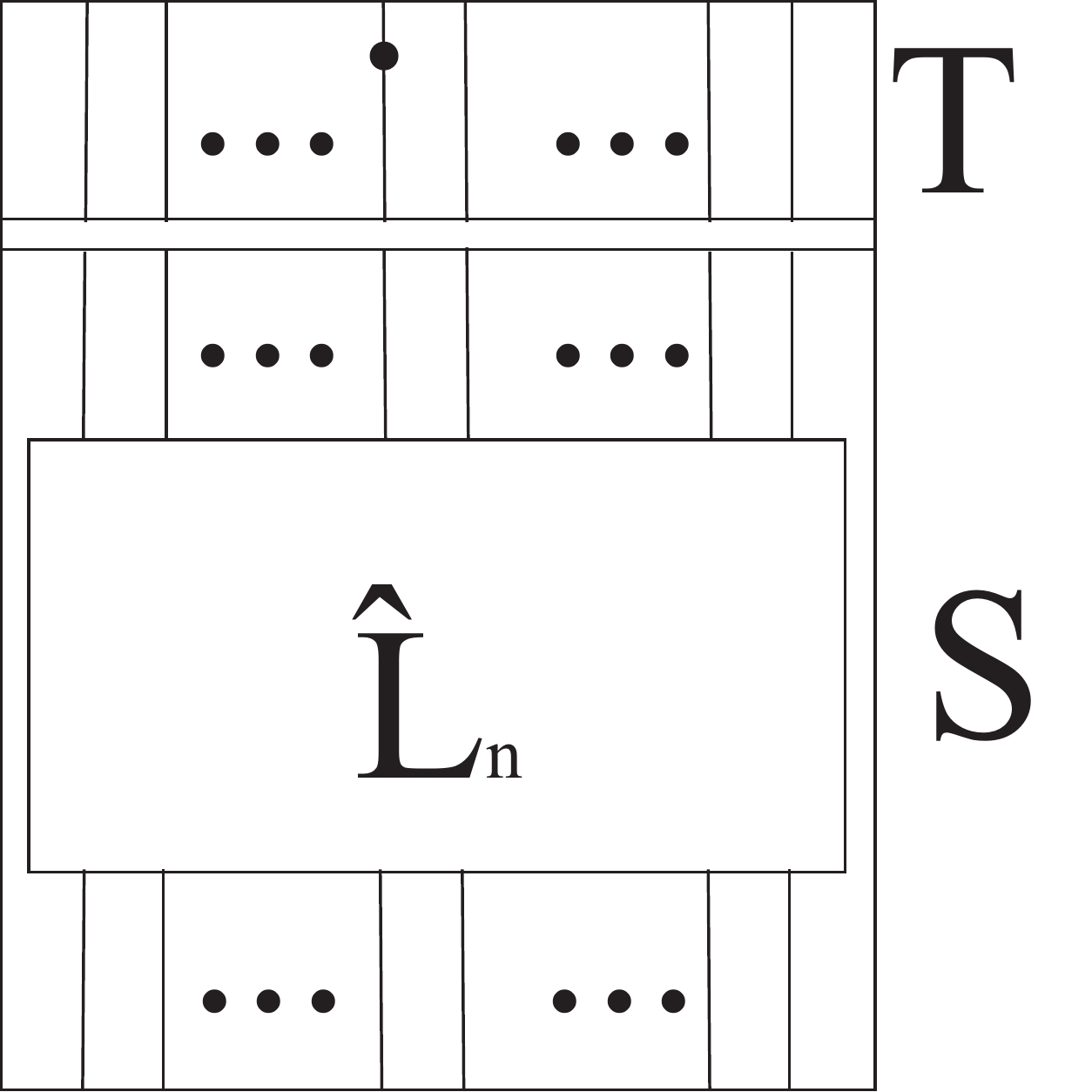}
\caption{Here $T$ is the trivial part with the black dot.}
\label{blackdot}
\end{figure}

Thus the SSE class of the tangle  endomorphism $Z(\mathcal{T}^n_n)$ constructed as above depends only on a surgery presentation $(D_0,L,s,G)$.
Thus we can denote this class by $Z(D_0,L,s,G)$.

\begin{thm}
\label{main}
Let $(D_0,L_1,s_1,G_1)$ and $(D_0,L_2,s_2,G_2)$ be  two surgery presentations for  $((M,r),\chi,G')$,
an extended closed $3$-manifold with  an embedded  colored trivalent graph, and choice of infinite cyclic cover.
Then $$Z(D_0,L_1,s_1,G_1)=Z(D_0,L_2,s_2,G_2).$$
Thus we may denote this SSE class by $Z((M,r),\chi,G')$.
\end{thm}

This theorem will be proved in section \ref{PMT}, after the way has been prepared in sections  \ref{TVM} and, \ref{relation}.

\section{The Turaev-Viro endomorphism}
\label{TVM}

In \S  \ref{introduction}, we introduced the basic idea of the Turaev-Viro endomorphism.
In this section, we will include the technical  details.

\begin{rem}
The discussion in this section and the next  section works for $3$-manifolds with an embedded $p$-admissibly colored trivalent graph.
For simplicity, we usually omit mention of  the trivalent graph. Thus we will write $((M,r), \chi)$ instead of $((M,r), \chi, G')$. 
This is according to the philosophy that we should think of the  colored trivalent graph $G'$ as simply some extra structure on $M$.
\end{rem}

\begin{lem}
\label{lagrangianchange}
Let $(M,r,\lambda(\partial M)_1)$ be an extended cobordism from $(\Sigma,\lambda(\Sigma)_1)$ to itself and $(M,r,\lambda(\partial M)_2)$ be an extended cobordism from $(\Sigma,\lambda(\Sigma)_2)$ to itself.
Then $Z((M,r,\lambda(\partial M)_1))$ is strong shift equivalent to $Z((M,r,\lambda(\partial M)_2))$.
\end{lem}
\begin{proof}
First we notice that
\begin{eqnarray}
&&\lambda(\partial M)_1=\lambda(\Sigma)_1\oplus\lambda(\Sigma)_1\in H_1(\Sigma)\oplus H_1(\bar{\Sigma}),
\notag\\
&&\lambda(\partial M)_2=\lambda(\Sigma)_2\oplus\lambda(\Sigma)_2\in H_1(\Sigma)\oplus H_1(\bar{\Sigma}).
\notag
\end{eqnarray}
Then we have
\begin{eqnarray}
&&(M,r,\lambda(\partial M)_1)\notag\\
&&=(\Sigma\times I,0,\lambda(\Sigma)_1\oplus\lambda(\Sigma)_2)\cup_{(\Sigma,\lambda(\Sigma)_2)}(M,r,\lambda(\Sigma)_2\oplus\lambda(\Sigma)_1),
\notag\\
&&(M,r,\lambda(\partial M)_2)\notag\\
&&=(M,r,\lambda(\Sigma)_2\oplus\lambda(\Sigma)_1)\cup_{(\Sigma,\lambda(\Sigma)_1)}(\Sigma\times I,0,\lambda(\Sigma)_1\oplus\lambda(\Sigma)_2).
\notag
\end{eqnarray}
Here we consider $(M,r,\lambda(\Sigma)_2\oplus\lambda(\Sigma)_1)$ as a cobordism from $(\Sigma,\lambda(\Sigma)_2)$ to $(\Sigma,\lambda(\Sigma)_1)$.
Then by the functoriality of $Z$, we have the conclusion. 
\end{proof}

\begin{lem}
\label{surfacechange}
Suppose we have a closed extended 3-manifold $((M,r),\chi)$ with an infinite cyclic covering.
We obtain two extended fundamental domains $M_1$ and $M_2$ by slicing along two extended surfaces $(\Sigma,\lambda(\Sigma))$ and $(\Sigma',\lambda(\Sigma'))$ which are dual to $\chi$.
We obtain two morphisms
\begin{eqnarray}
&&(M_1,r_1,\lambda(\Sigma)\oplus\lambda(\Sigma)):(\Sigma,\lambda(\Sigma))\rightarrow(\Sigma,\lambda(\Sigma)),
\notag\\
&&(M_2,r_2,\lambda(\Sigma')\oplus\lambda(\Sigma')):(\Sigma',\lambda(\Sigma'))\rightarrow(\Sigma',\lambda(\Sigma')),
\notag
\end{eqnarray}
with weight $r_1,r_2$ respectively such that the closures of both cobordism having weight $r$.
Then 
\begin{equation}
Z((M_1,r_1,\lambda(\Sigma)\oplus\lambda(\Sigma)))\sim Z((M_2,r_2,\lambda(\Sigma')\oplus\lambda(\Sigma')).
\notag
\end{equation}
\end{lem}
\begin{proof}
We just need prove the case where $\Sigma$ and $\Sigma'$ are disjoint from each other.
See \cite[Proof of Theorem 8.2]{L1}, \cite{G3}.
Since $\Sigma'$ is disjoint from $\Sigma$,
we can choose a copy of $(\Sigma',\lambda(\Sigma'))$ inside $(M_1,r_1,\lambda(\Sigma)\oplus\lambda(\Sigma))$.
We cut along $\Sigma'$ and get two $3$-manifolds $T,S$.
We assign to $T,S$ extended $3$-manifold structures, denoted by $(T,t,\lambda(\Sigma)\oplus\lambda(\Sigma'))$ and $(S,s,\lambda(\Sigma')\oplus\lambda(\Sigma))$,
such that if we glue $R$ to $S$ along $\Sigma'$, we get $(M_1,r_1,\lambda(\Sigma)\oplus\lambda(\Sigma))$ back.
We need to choose appropriate weights $t,s$ for $T,S$.
Using Definition  \ref{extendedgluing}, we see that such $t,s$ exists.
Now we just need prove that if we glue $S$ to $T$ along $\Sigma$,
we obtain $(M_2,r_2,\lambda(\Sigma')\oplus\lambda(\Sigma'))$.
Actually, it is easy to see that after gluing, we have the right base manifold and lagrangian subspace.
What we need to prove is that we get the right weight.
This follows from Lemma \ref{closure}.
\end{proof}

As a consequence of the two lemmas above, we have the following:

\begin{prop}
\label{invariant}
For a tuple $((M,r),\chi)$ and $(M_1,r_1,\lambda(\Sigma_1)\oplus\lambda(\Sigma_1))$ given as in Lemma \ref{surfacechange},
the strong shift equivalent class of the map $Z((M_1,r_1,\lambda(\Sigma_1)\oplus\lambda(\Sigma_1)))$ is independent of the choice of the extended surface $(\Sigma_1,\lambda(\Sigma_1))$.
Thus we may denote this SSE class by $\mathcal{Z}((M,r),\chi)$.
\end{prop}

Next, we work towards constructing a fundamental domain for  an extended $3$-manifold $((M,r),\chi)$ with an infinite cyclic covering.
Suppose we have a surgery presentation $(D_0,L,s)$ in standard form for $((M,r),\chi)$, here $s=r-\sigma(L)$ \cite[Lemma(2.2)]{GM}.
We do $0$-surgery along $K_0$ and get a link $L$ in $(\sts,s)$.
We cut $\sts$ along the  2-sphere containing $D_0$ in this product structure and obtain a tangle $T$ in $(S^2\times I,s)$ in standard form. Here, we say that a tangle is in standard form if it comes from slicing a surgery presentation diagram  in standard form.
Then we drill out tunnels along arcs which meet the bottom and glue them back to the corresponding place on the top.
We obtain a  cobordism $\hat E$ from $\Sigma_g$ to itself with a link $\hat{L}$  embedded in it as in Figure \ref{fundamentaldomain},
where $\Sigma_g$ is a genus $g$ closed surface. 
See \cite[Figure 3]{G1} for example.
Moreover, we identify $\Sigma_g$ with a standard surface as pictured in Figure \ref{standardsurface}.
We denote   by $\lambda_{\mathcal{A}}$ the lagrangian subspace spanned by the  curves 
labelled by $a_i$  in Figure \ref{standardsurface}.
We assign the lagrangian subspace $\lambda_{\mathcal{A}}$ to each connected component of the boundary of $\hat{E}$.
Moreover, we  assign the weight $s$ to it.
Thus we obtain an extended cobordism $(\hat{E},s,\lambda_\mathcal{A}\oplus\lambda_{\mathcal{A}})$.

\begin{figure}[h]
\includegraphics[width=1.2in,height=1.2in]{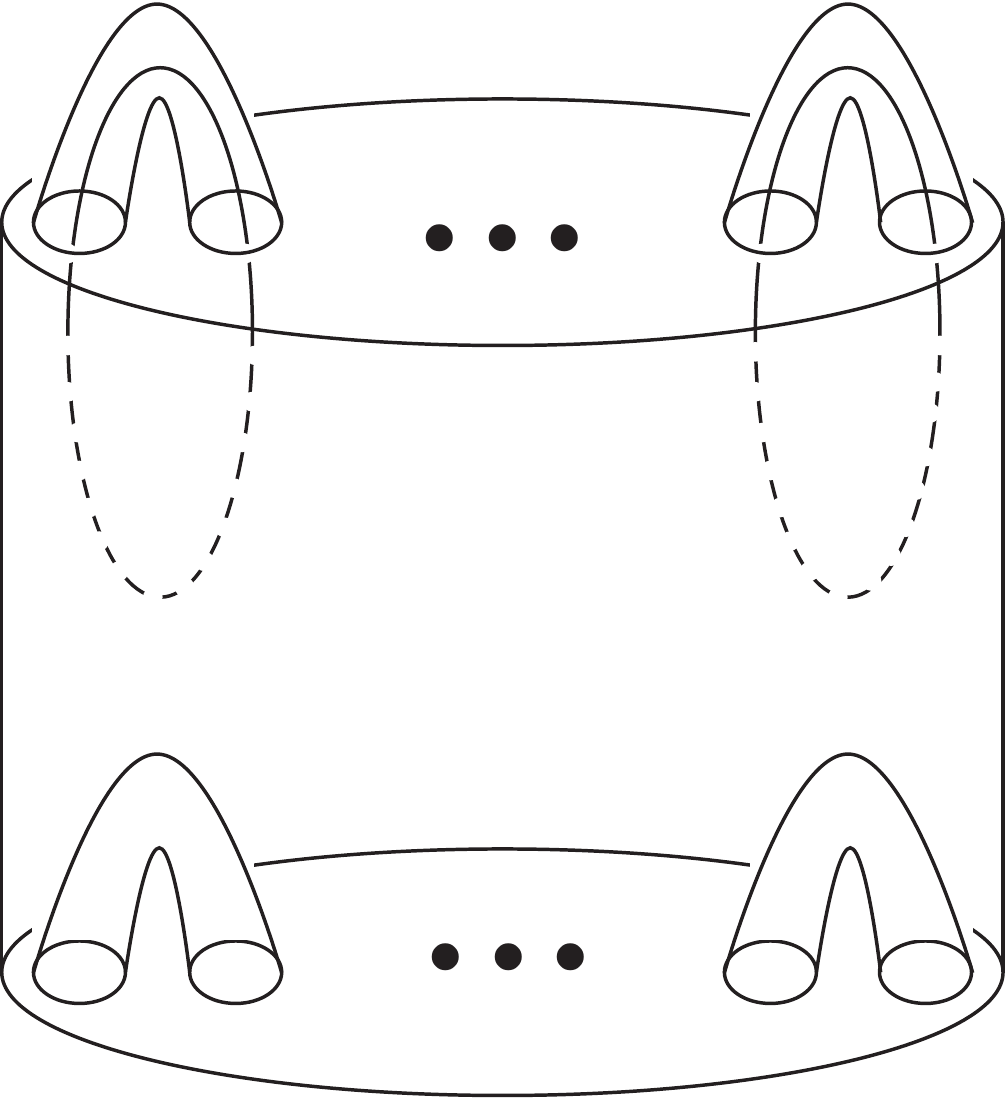}
\caption{The extended cobordism $(\hat{E},s,\lambda_\mathcal{A}\oplus\lambda_{\mathcal{A}})$  containing a framed link $\hat L$. If we do extended surgery along $\hat L$, we get a fundamental domain $E$. 
If, instead,  we color  $\hat L$ by $\omega$, we obtain another cobordism $E'$.}
\label{fundamentaldomain}
\end{figure}

\begin{figure}
\includegraphics[width=1.2in,height=0.6in]{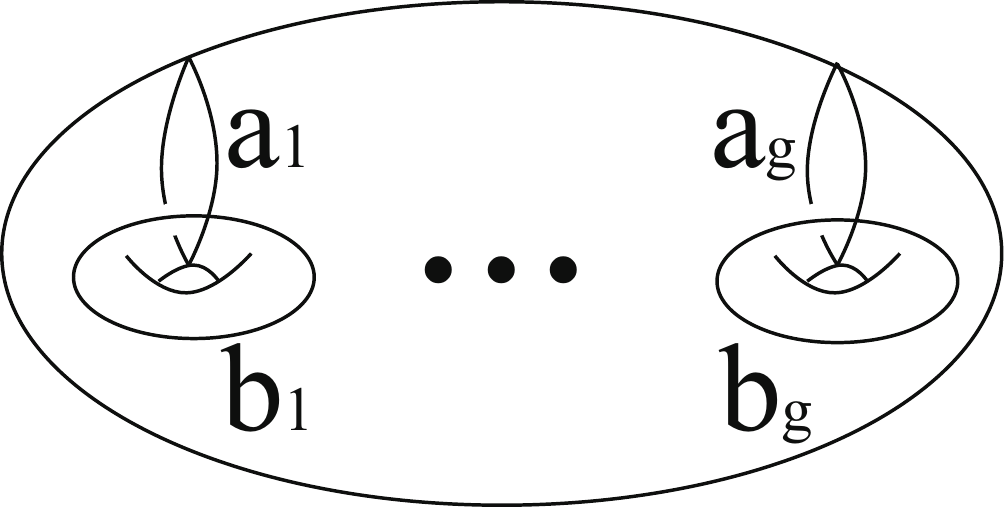}
\caption{A surface in standard position.}
\label{standardsurface}
\end{figure}

\begin{prop}
\label{weight}
The closure of $(\hat{E},s,\lambda_{\mathcal{A}}\oplus\lambda_{\mathcal{A}})$ is $(S^3(U),s,0)$,
where $U$ is a $0$-framed unknot.
\end{prop}
\begin{proof}
It is easy to see that the closure of $\hat{E}$ is $S^3(U)$.
Then we just need to prove that the weight of the closure is $s$.
By the gluing formula and Definition \ref{selfgluing}, 
we have that the weight on $S^3(U)$ is
\begin{equation} 
s+0-\mu(\lambda_{E}(\Sigma_g\cup\bar{\Sigma}_g),\lambda(\Sigma_g\cup\bar{\Sigma}_g),\lambda_{\Sigma_g\times[0,1]}(\Sigma_g\cup\bar{\Sigma}_g)).
\notag
\end{equation}
Now let
\begin{equation}
H_1(\Sigma_g)=<a_1,\cdots,a_g,b_1,\cdots,b_g>,H_1(\bar{\Sigma}_g)=<a'_1,\cdots,a'_g,b'_1,\cdots,b'_g>.
\notag
\end{equation}
Then
\begin{eqnarray}
\lambda_{E}(\Sigma_g\cup\bar{\Sigma}_g)&=&i_{\Sigma_g\cup\bar{\Sigma}_g,E}^{-1}(0)
\notag\\
&=&\{(x,y)\mid x\in<a_1,\cdots,a_g>,y\in<b'_1,\cdots,b'_g>\}.
\notag\\
\lambda_{\Sigma_g\times[0,1]}(\Sigma_g\cup\bar{\Sigma}_g)&=&i_{\Sigma_g\cup\bar{\Sigma}_g,\Sigma_g\times[0,1]}^{-1}(0)
\notag\\
&=&<(a_i,-a'_i),(b_i,-b'_i)\mid i=1,\cdots,g>.
\notag\\
\lambda(\Sigma_g\cup\bar{\Sigma}_g)&=&\lambda_{\mathcal{A}}\oplus\lambda_{\mathcal{A}}
\notag\\
&=&\{(x,y)\mid x\in<a_1,\cdots,a_g>, y\in<a'_1,\cdots,a'_g>\}
\notag
\end{eqnarray}
So
\begin{eqnarray}
&&\lambda(\Sigma_g\cup\bar{\Sigma}_g)+\lambda_{E}(\Sigma_g\cup\bar{\Sigma}_g)
\notag\\
&=&\{(x,y)\mid x\in<a_1,\cdots,a_g>,y\in<a'_1,\cdots,a'_g>+<b'_1,..,b'_g>\}
\notag\\
&=&\{(x,y)\mid x\in<a_1,\cdots,a_g>,y\in H_1(\bar{\Sigma}_g)\}.
\notag
\end{eqnarray}
Therefore,
\begin{eqnarray}
&&\lambda_{\Sigma_g\times[0,1]}(\Sigma_g\cup\bar{\Sigma}_g)\cap[\lambda(\Sigma_g\cup\bar{\Sigma}_g)+\lambda_{E}(\Sigma_g\cup\bar{\Sigma}_g)]
\notag\\
&=&<(a_i,-a'_i)\mid i=1,\cdots,g>.
\notag
\end{eqnarray}
It is easy to see that the bilinear form defined in \cite{Wall} is identically $0$ on $<(a_i,-a'_i)\mid i=1,\cdots,g>$.
So we have 
\begin{equation}
\mu(\lambda_{E}(\Sigma_g\cup\bar{\Sigma}_g),\lambda(\Sigma_g\cup\bar{\Sigma}_g),\lambda_{\Sigma_g\times[0,1]}(\Sigma_g\cup\bar{\Sigma}_g))=0.
\notag
\end{equation}
Then we get the conclusion.
\end{proof}

\begin{prop} 
Let $(E,s,\lambda_{\mathcal{A}}\oplus\lambda_{\mathcal{A}})$  be the result of  extended surgery along the embedded link $\hat{L}$ in 
$(\hat{E},s,\lambda_{\mathcal{A}}\oplus\lambda_{\mathcal{A}})$ constructed as above starting with  a standard surgery presentation diagram  for $((M,r),\chi)$.   $(E,s,\lambda_{\mathcal{A}}\oplus\lambda_{\mathcal{A}})$ is a fundamental domain for $((M,r),\chi)$. \end{prop} 

\begin{proof}The closure of $(E,s,\lambda_{\mathcal{A}}\oplus\lambda_{\mathcal{A}})$ can be obtained by performing extended surgery on the closure of $(\hat{E},s,\lambda_{\mathcal{A}}\oplus\lambda_{\mathcal{A}})$.
This uses the commutative property of gluing discussed in \cite{GM}. 
Thus the closure of $E$ is diffeomorphic to $M$, and by \cite[Lemma 2.2]{GM},
we see that the closure of $(E,s,\lambda_{\mathcal{A}}\oplus\lambda_{\mathcal{A}})$ has weight $r$. 
\end{proof}.

\begin{prop} 
Let $(E',s,\lambda_{\mathcal{A}}\oplus\lambda_{\mathcal{A}})$ be the extended cobordism obtained 
by coloring the link $\hat{L}$ in $(\hat{E},s,\lambda_{\mathcal{A}}\oplus\lambda_{\mathcal{A}})$ by $\omega$.
The SSE class $\mathcal{Z}((M,r),\chi)$  is given by
$$Z(E',s,\lambda_{\mathcal{A}}\oplus\lambda_{\mathcal{A}}).$$
\end{prop}
\begin{proof}
The equality $Z(E,s,\lambda_{\mathcal{A}}\oplus\lambda_{\mathcal{A}})=Z(E',s,\lambda_{\mathcal{A}}\oplus\lambda_{\mathcal{A}})$ follows from the surgery axiom  \cite[Lemma 11.1]{GM} for extended surgery.
\end{proof}


\section{The relation between the Turaev-Viro endomorphism and  the tangle endomorphism.}
\label{relation}

In this section, we will prove the following theorem. 

\begin{thm}
\label{TVisT}
If $((M,r),\chi)$ is an extended $3$-manifold with an  infinite cyclic covering having a surgery presentation $(D_0,L,s)$ in standard form,
then  $\mathcal{Z}((M,r),\chi)= Z(D_0,L,s)$.
\end{thm}
\begin{proof}
For simplicity, we indicate the proof in case that $((M,r),\chi)$ does not have a colored trivalent graph. The argument may easily be adapted to the more general case.

We obtain a tangle  $\hat{L}_n$ from the surgery presentation $(D_0,L,s)$, 
and we place black dots on  segments in the top part. 
We will directly compute two matrices for these two endomorphisms with respect to some bases.

{\bf Step 1: Compute the entry for the Turaev-Viro endomorphism.}
We will use the basis in \cite{BHMV} for $V(\Sigma_g)$,
where $\Sigma_g$ is genus $g$ surface.
Specifically we choose our spine to be a lollipop graph, as in \cite{GM2}.
We show one example of elements as in Figure \ref{basis1}.
\begin{figure}[h]
\includegraphics[width=1.4in,height=1.2in]{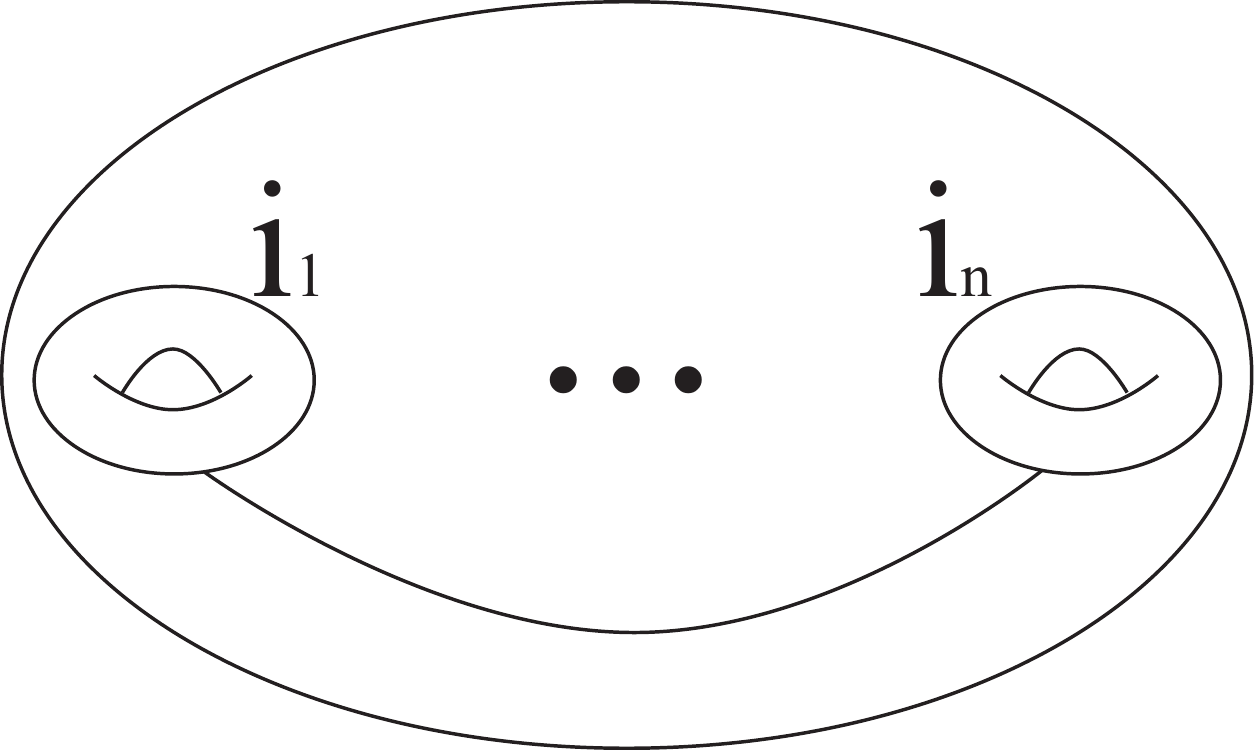}
\caption{An example of elements in the basis  for $V(\Sigma_g)$ constructed in \cite{BHMV}.}
\label{basis1}
\end{figure}

Using the method employed in \cite[\S 8]{G2}, we can compute the entries of the matrix,
with respect to this basis by computing the quantum invariants of   colored links in a connected sum of $S^1 \times S^2$'s . We have
\begin{eqnarray}
&&(i_1,\cdots,i_g)\text{-}(j_1,\cdots,j_g) \text{\ entry of\ }Z((E',s,\lambda_{\mathcal{A}}\oplus\lambda_{\mathcal{A}})\notag\\
&=&\frac{\eta\kappa^{s-\sigma(L')}<\text{the first diagram in Figure \ref{tventry}}>}{\eta\kappa^{-\sigma(L'')}<\text{the first diagram in Figure \ref{JS}}>}.\notag
\end{eqnarray}
where $L'$ as in Figure \ref{tventry} and $L''$ as in Figure \ref{JS}.
\begin{figure}[h]
\includegraphics[width=1in,height=1.2in]{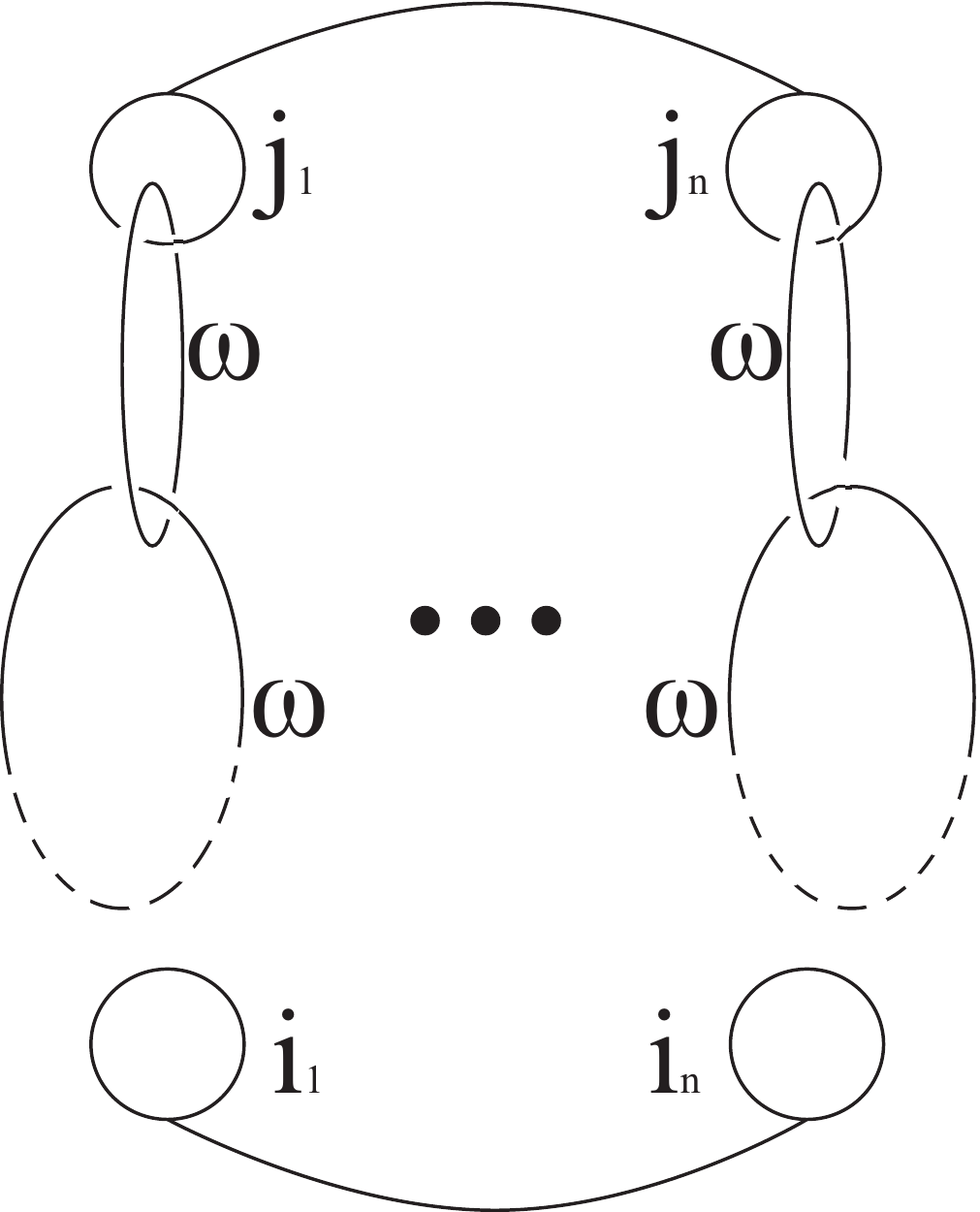}\hskip 0.5in
\includegraphics[width=1in,height=1.2in]{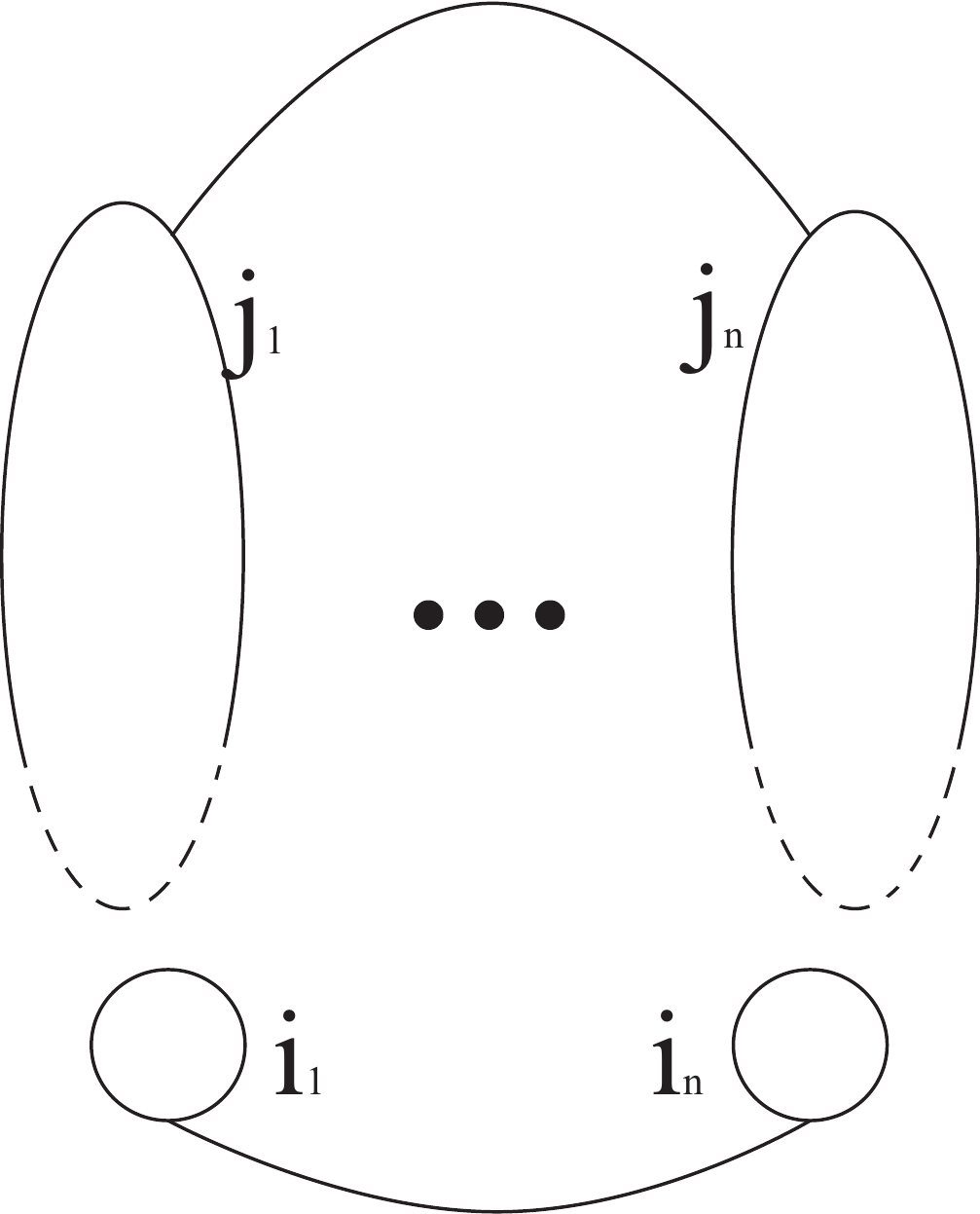}
\caption{$L'$ is consisted of components colored with $\omega$.}
\label{tventry}
\end{figure}
\begin{figure}[h]
\includegraphics[width=1in,height=0.8in]{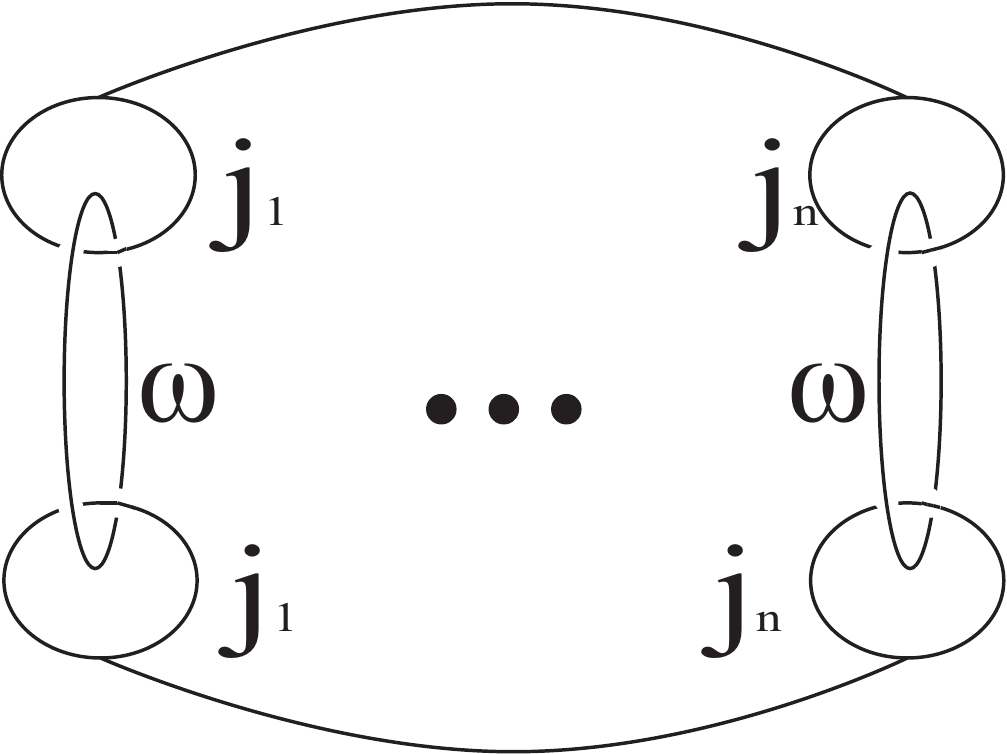}\hskip 0.5in
\includegraphics[width=1in,height=0.8in]{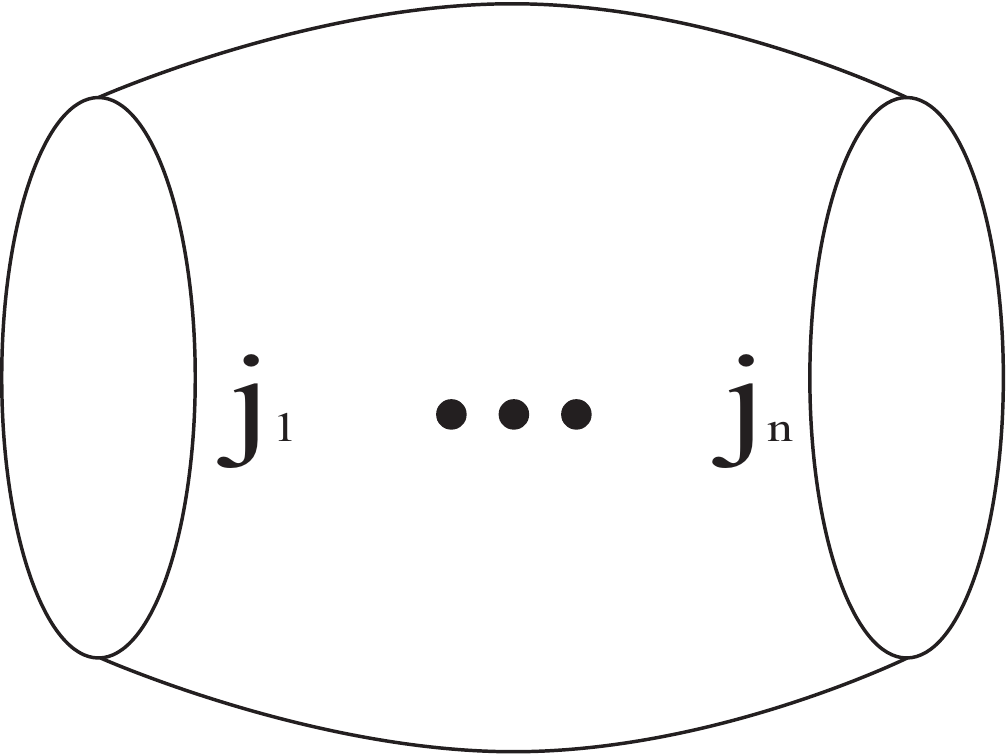}
\caption{$L''$ is consisted of components colored with $\omega$.}
\label{JS}
\end{figure}

By using fusion and Lemma 6 in \cite{L2} and the fact that
\begin{equation}
\sigma(L')=\sigma(L'')=0,\notag
\end{equation}
we have
\begin{eqnarray}
&&(i_1,\cdots,i_g)\text{-}(j_1,\cdots,j_g)\text{\ entry of\ }Z((E',-\sigma(L),\lambda_{\mathcal{A}}\oplus\lambda_{\mathcal{A}})\notag\\
&=&\frac{\eta\kappa^{s}\eta^n\Delta_{j_1}\cdots\Delta_{j_n}<U(\omega)>^n<\text{the second diagram in Figure \ref{tventry}}>}{\eta<U(\omega)>^n<\text{the second diagram in Figure \ref{JS}}>}\notag\\
&=&\frac{\kappa^{s}\eta^n\Delta_{j_1}\cdots\Delta_{j_n}<\text{the second diagram in Figure \ref{tventry}}>}{<\text{the second diagram in Figure \ref{JS}}>}\notag
\notag
\end{eqnarray}
where $U(\omega)$ is the $0$-framing unknot colored with $\omega$.

{\bf Step 2: Compute the entry for tangle endomorphism.}
\begin{figure}[h]
\includegraphics[width=1.4in,height=1.4in]{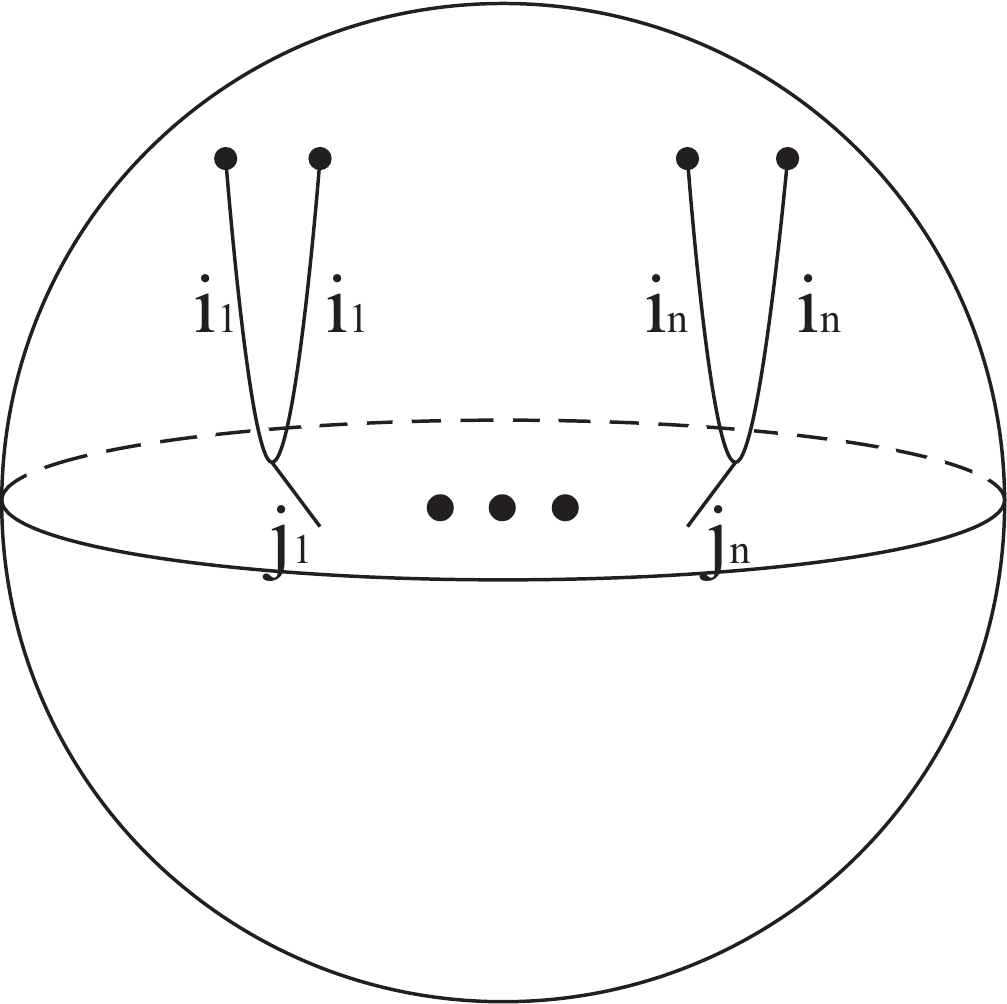}
\caption{Elements in a basis for ${V(S^2;2 n)}$ which do not automatically vanish under $Z(\hat{L}_n)$ have this form.}
\label{basis4}
\end{figure}

By gluing the tangle in $(S^2\times I,s)$ to the basis element in Figure \ref{basis4}, we can see that
\begin{eqnarray}
&&(i_1,\cdots,i_g)\text{-}(j_1,\cdots,j_g)\text{\ entry of\ }Z(\hat{L}_n)\notag\\
&=&\frac{\kappa^{s}\eta^n\Delta_{j_1}\cdots\Delta_{j_n}<\text{the second diagram in Figure \ref{tventry}}>}{<\text{the second diagram in Figure \ref{JS}}>}\notag
\notag
\end{eqnarray}
for $(i_1,\cdots,i_n,j_1,\cdots,j_n)$ a legal coloring, and is zero  otherwise.

{\bf Step 3: The two matrices are strong shift equivalent.}
By above discussion,
it is easy to see that if the matrix for Turaev-Viro endomorphism is $X$,
then a matrix for tangle endomorphism is the block matrix
$
\begin{bmatrix}
X & 0\\
0 & 0
\end{bmatrix}.
$
 We see that this  block matrix is  strong shift equivalent  to $X$ by Proposition \ref{ker}.
\end{proof}


\section{Proof of  Theorem \ref{main}}
\label{PMT}

\begin{lem}
\label{transforminvariant}
The transformation process in Proposition \ref{standardization} does not change the strong shift equivalent class of the tangle endomorphism.
\end{lem}
\begin{proof}
A small 
 extended
 Kirby-$1$ move adds a $\pm 1$ framed $\omega$ to all the different decorations of $\mathcal{T}^n_n$
which go into the definition of  $Z(\mathcal{T}^n_n)$. This would seem to multiply $Z(\mathcal{T}^n_n)$  by $\kappa^{\pm1}$. But a small extended Kirby-$1$ move also changes $\sigma(L)$ by $\pm 1,$ and thus changes the weight $s$ of $S^2 \times I \supset \mathcal{T}^n_n$ by $\mp 1$.  These two effects of the move cancel out and $Z(\mathcal{T}^n_n)$ is unchanged. 
The small Kirby-$2$ moves preserves all the summands of $Z(\mathcal{T}^n_n)$, by a well known handle slide property of $\omega.$ 
See \cite[Lemma 21]{KL} for instance.
Two tangles related by a  $D_0$ move are obtained  by cutting $\sts$ along two different $S^2$'s.
Suppose if we cut $\sts$ along $S_0=\stpo$, we obtain a tangle $\hat{L}_n$.
If we cut along $S_1=\stpl$, we obtain a tangle $\hat{L}'_m$.
By those two cutting, we obtain two homomorphisms
\begin{equation}
Z(\hat{L}_n):\vsn\rightarrow\vsn
\notag
\end{equation}
and
\begin{equation}
Z(\hat{L}'_m):\vsm\rightarrow\vsm.
\notag
\end{equation}
Now suppose we cut $\sts$ along $\stpo$ and $\stpl$,
we get a $(n,m)$-tangle in $\sti$, denoted by $T_1$, 
and a $(m,n)$-tangle, denoted by $T_2$.
$T_1$ defines a  homomorphism
\begin{equation}
Z(T_1):\vsn\rightarrow\vsm,
\notag
\end{equation}
and $T_2$ defines a homomorphism
\begin{equation}
Z(T_2):\vsm\rightarrow\vsn.
\notag
\end{equation}
It is easy to see that
\begin{equation}
Z(\hat{L}_n)=Z(T_2)Z(T_1),
\notag
\end{equation}
and
\begin{equation}
Z(\hat{L}'_m)=Z(T_1)Z(T_2).
\notag
\end{equation}
Therefore, $Z(\hat{L}_n)$ is strong shift equivalent to $Z(\hat{L}'_m).$
\end{proof}

\begin{lem}
\label{TV}
Suppose we have two surgery presentations $(D_0,L_1,s_1,G_1)$ and $(D_0,L_2,s_2,G_2)$ for $(M,r,\chi,G')$ in standard form,
then 
\begin{equation}
Z(D_0,L_1,s_1,G_1)=  Z(D_0,L_2,s_2,G_2).
\notag
\end{equation}
\end{lem}
\begin{proof}
This easily follows from Propositions \ref{TVisT}.
\end{proof}

\begin{proof}[Proof of Theorem \ref{main}]
By Proposition \ref{standardization}, and Lemma \ref{transforminvariant}, we can transform $(D_0,L_1,s_1,G_1)$ and $(D_0,L_2,s_2,G_2)$ so that they are standard without changing the SSE class of their induced tangle endomorphism.
Then the result follows from Lemma \ref{TV}.
\end{proof}


\section{Colored Jones polynomials and Turaev-Viro endomorphisms}
\label{CJ}

In this section, we assume, for simplicity, that $p$ is odd. Similar formulas could be given for $p$ even, by the same methods.
We let $J(K,i)$ denote the bracket evaluation of a knot diagram of $K$ with zero writhe colored $i$ at a primitive $2p$th root of unity $A$. Letting  $U$ denote the unknot, we have that $J(U,i) = \Delta_i$. In particular,  $J(U,1) = -A^2-A^{-2}$.  This is one normalization of the colored Jones polynomial  at a root of unity.

\begin{rem} \label{sym} 
 Using \cite[Lemma 6.3]{BHMV1}, we have that: $$ J(K, i+p)=  -J(K, i)  \text{,  and }
J(K, i+(p-1)/2)= J(K, -i+(p-3)/2).$$
Without losing information,  we can restrict our attention to $J(K,2i)$ for $0 \le i \le (p-3)/2.$  For other $c$,  $J(K,c) = \pm J(K,2i)$ for some $0 \le i \le (p-3)/2,$  
using the above equations.\end{rem}

Let $(S^3(K),i,j,0))$ denote 0-framed surgery along an oriented knot $K$ in $S^3$ decorated with a meridian to $K$ colored $i$ and a longitude  little further away from $K$
 colored $j$ and equipped with the weight zero. Let $\chi$ be the homomorphism from $H_1(M)$  to $\bz$ which sends a meridian to one. 
Let $\TV(K,i,j)$ denote  
the SSE class of 
 the Turaev-Viro endomorphism $\mathcal{Z}(S^3(K),i,j,0)),\chi)$.  The vector space associated to a 2-sphere with just one colored point which is  colored by an odd number is zero. Using this fact, and a surgery presentation, one sees that 
$$\TV(K,i,j)=0 \text{ if $i$ is odd.}$$
The second author  studied $\TV(K,i,0)$ \cite{G1,G2}. The idea of adding the longitude with varying colors is due to Viro \cite{Viro',Viro}.   The least interesting case, of this next theorem, when $j=0$ 
 already appeared in \cite[Corollary 8.3]{G1}.   

\begin{thm}[Viro]
\label{ColoredJonesequalTraceSum} 
For $0 \le j \le p-2$, 
\begin{equation*}\label{J}
J(K, j)= \sum_{i=0}^{(p-3)/2} \Delta_{2i} \Trace(\TV(K,2i,j)).
\end{equation*}
\end{thm}

\begin{proof} One has that $0$-framed surgery along $K$ with the weight zero is the result of extended surgery of $S^3$ with weight zero along a zero-framed copy of $K$.  
If we add then a zero-framed meridian of $K$ to this framed link description,  
we undo the surgery along $K$ and we get back an extended surgery description of $S^3$, also with weight zero.
A longitude to $K$ colored $j$ and  placed a little outside the meridian will go to a longitude of $K$ colored $j$ in $S^3$, which is of course isotopic to $K$. 
But adding a zero-framed meridian to the framed link changes $< \ >_p$ in the same way as cabling by $\omega = \eta\sum_{i=0}^{(p-3)/2} \Delta_{2i} e_{2i}$. 
If we cable the meridian of $K$  by $e_{2i}$ instead of by $\omega$, 
and calculate 
$< \ >_p$, 
we get 
$$<(S^3(K),2i,j,0)>_p=\Trace (\TV(K,2i,j)),$$
by the trace property of TQFT \cite[1.2]{BHMV}.  
Thus 
$$<S^3 \text{ with $K$ colored $j$} >= \eta \sum_{i=0}^{(p-3)/2} \Delta_{2i} \Trace (\TV(K,2i,j)).$$ 
Dividing by $\eta$ yields the result. 
\end{proof}

Thus the colored Jones is determined by the traces of the  $\TV(K,2i,j).$ The next theorem shows that the $J(K,j)$ determine the traces of the $\TV(K,2i,j)$.

\begin{thm}
\label{TraceequalColoredJonesSum}
For $0 \le i, j \le (p-3)/2$, 
\begin{equation*}\label{Trace}
\Trace(\TV(K,2i,2j))=  \eta^2  \sum_{k=0}^{(p-3)/2}   \sum_{l= |k-j|}^{k+j} \Delta_{(2k+1)(2i+1)-1} J(K,2l).
\end{equation*}
More generally :
$$\Trace(\TV(K,2i,j))=  \eta^2  \sum_{k=0}^{(p-3)/2}   \sum_{
\begin{matrix} 
l= |2k-j| \\
l \equiv j \mod 2
\end{matrix}
}
^{2k+j} \Delta_{(2k+1)(2i +1)-1} J(K,l).$$.
\end{thm}

\begin{proof} 
By the trace property of TQFT, 
$$\Trace(\TV(K,2i,2j))= <(S^3(K),2i,2j)>_p.$$
Direct calculation of $<(S^3(K),2i,2j)>_p$ from the definition yields $\eta$ times the bracket evaluation of $K$ cabled by $\omega$
together with the meridian  colored $2i$ and the longitude further out colored $2j$.  These skeins all lie in a regular neighborhood of $K$ with framing zero. These skeins can then be expanded as a linear combination of  the core of this solid torus with different colors.

The operation of encircling an arc colored $2k$ with loop  colored $2j$ in the skein module of a local disk has the same effect as multiplying the arc by  $\Delta_{(2k+1)(2j+1)-1}/\Delta_{2k}$ by
\cite[Lemma 14.2]{L1}. Note the idempotents $f_{k}$ are only defined for $0 \le k \le (p-2).$ 
It is well known that the $e_k$ satisfy  a recursive formula  which can be used to extend the definition of $e_k$ for all $k \ge 0$. This is given \cite{BHMV1} as follows:
$e_0=1$, $e_1$ is the zero framed core of a solid standard solid torus, and $e_k= z e_{k-1}- e_{k-2}$.
  In the skein module of a  solid torus,  we  have  $e_{2k}.e_{2j} = \sum_{l= |k-j|}^{k+j} e_{2l}$. 
Using these rules,  the expansion can be worked out to be  $$\eta  \sum_{k=0}^{(p-3)/2}   \sum_{l= |k-j|}^{k+j} \Delta_{(2k+1)(2i+1)-1} e_{2l}.$$

The second equation is worked out in a similar way.
\end{proof}

Notice that,  in the summation on the right of the first equation in Theorem \ref{Trace}, $J(K,2l)$ for $l > {(p-3)/2}$ sometimes appears. 
This can be rewritten using Remark \ref{sym} as $J(K,2j)$ for $j \le {(p-3)/2}$.

We remark that using \cite[Corollary 2.8]{G4},  one can see that the Turaev-Viro polynomials 
 of $\TV(K,i,j)$ will have coefficients in a cyclotomic ring of integers, if $p$ is an odd prime or twice an odd prime.


\section{Examples}
\label{example}

In this section, we wish to illustrate with some concrete examples how to calculate the $\TV(K,i,j)$ using tangle morphisms in the case $p=5$ (which is the first interesting case). 
For both examples, we check our computation against an identity from the previous section.

The first example  is the $k$-twist knot with meridian colored $0$ or $2$ and longitude colored $2$.  We then  verify
directly the equation in Theorem \ref{ColoredJonesequalTraceSum} for the case $p=5$, $j=2$, and $K$ is  the $k$-twist knot.

The second example we study is the knot $6_2$ with the meridian and longitude uncolored.
We  work out, using tangle morphisms, the traces of the Turaev-Viro  endomorphism.  We then  verify the equation in Theorem \ref{TraceequalColoredJonesSum} when $p=5$, $i=j=0$, and $K= 6_2$.

We pick an orthogonal basis for the module associated with a 2-sphere with some points, and use this basis to work out the entries on the matrix for the tangle endomorphism coming from a surgery presentation. The bases are represented by colored trees in the 3-ball which meet the boundary in the colored points as in Figure \ref{basis4}. Here we will refer to these colored trees as basis-trees. Each entry is obtained as a certain quotient. The numerator is the evaluation as a colored fat graph in $S^3$  obtained from the tangle  closed off with the source basis-tree  at the bottom and the target basis-tree at the top. The denominator is the quotient as the  evaluation of the double of  the target basis element.
In both examples, we use a surgery presentation, with one surgery curve with framing $+1$.  Thus the initial weight of $S^3$, denoted $s$ above,  should be $-1$, so the weight of $S^3$ after the surgery is zero. This puts a factor of $\kappa^{-1}$ in front of the tangle endomorphism. There is also a uniform factor of $\eta$ coming from the single black dot on a strand with two endpoints.  We put this total factor of $\kappa^{-1} \eta$ in front.  We also have $\Delta_i$ prefactors where  $i$ is the color of  the strand  with the black dot, and these factor vary from entry to entry.

To simplify our formulas, when $p=5$, we use $\Tet$ to abbreviate $\Tet(2,2,2,2,2,2)$, $\Delta$ to abbreviate  $\Delta_1=\Delta_2$  and $\Theta$ to denote $\Theta(2,2,2)$.

\subsection{The Turaev-Viro endomorphism and the colored Jones polynomial of the $k$-twist knot.}

A tangle $T$ for the $k$-twist knot with meridian and longitude is given in Figure \ref{double}.

\begin{figure}[h]
\includegraphics[width=1.2in,height=1.2in]{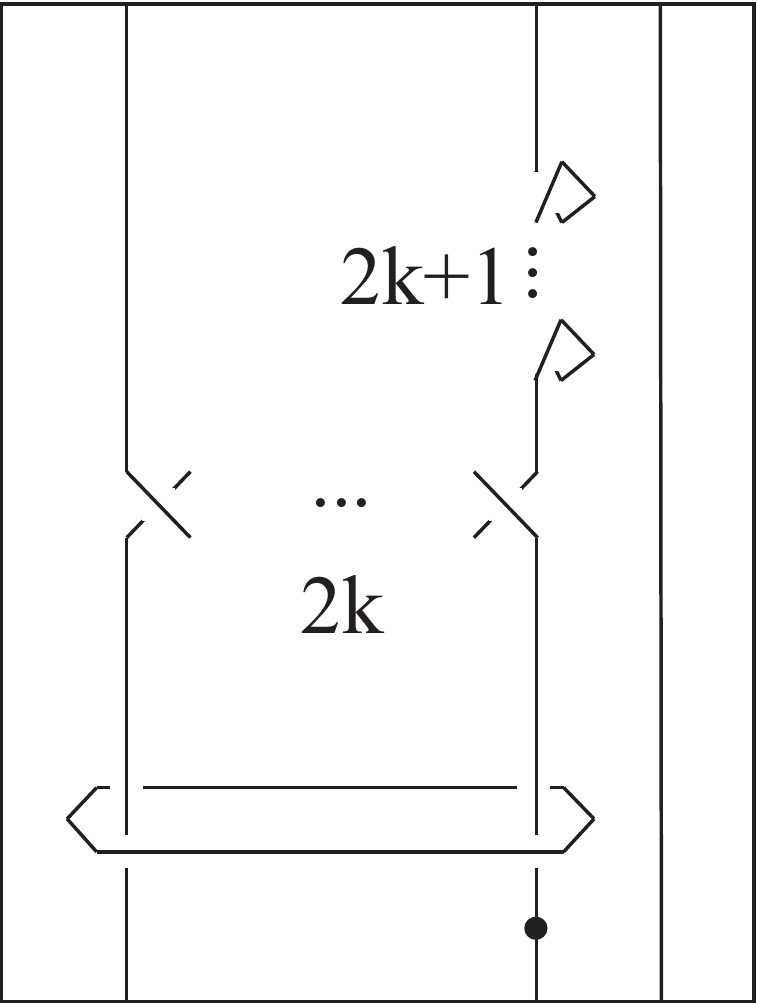}
\caption{Surgery presentation of $k$-twist knot with meridian and longitude.
The straight line is from the meridian and the circle is from the longitude. We have also chosen a position for the black dot.}
\label{double}
\end{figure}

If we denote by $T_0$ the tangle $T$ with meridian colored by $0$ and longitude colored by $2$, and let $S$ denote a 2-sphere with two uncolored points, 
then we obtain a map
\begin{equation}
\TV(K,0,2)=Z(T_0):\vs\rightarrow\vs.
\notag
\end{equation}
By using the trivalent graph basis in \cite{BHMV},
\begin{equation}
\vs=\text{Span}<a_1,a_2>,
\notag
\end{equation}
where $a_1,a_2$ are as in Figure \ref{basis0}.
\begin{figure}[h]
\includegraphics[width=0.8in,height=0.5in]{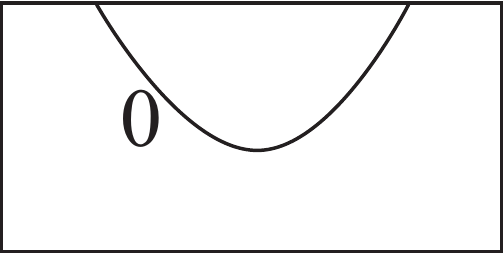}\hskip 0.5in
\includegraphics[width=0.8in,height=0.5in]{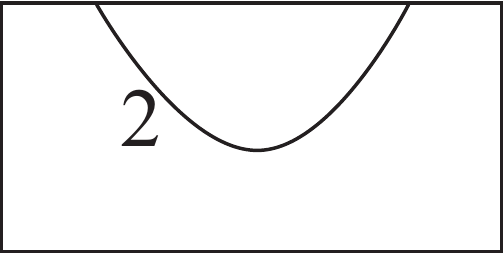}
\caption{A basis for $\vs$ where $S$ is a 2-sphere with two uncolored points}
\label{basis0}
\end{figure}
With respect to this basis,
we have
\begin{equation}
\TV(K,0,2)=\kappa^{-1}\eta
\begin{bmatrix}
   \Delta &  \Delta^3  \\
 \Delta A^{16k+8}  & \Delta (A^8+\Delta A^{8k+8}).
 \end{bmatrix}
\notag
\end{equation}

We follow the convention that the columns of the matrix for a linear transformation with respect to a basis are given  the images of that basis written in terms of that basis. The characteristic polynomial of this matrix  (i.e. the Turaev-Viro polynomial) has coefficients in $\bz[A].$

If we denote by $T_2$ the tangle $T$ with meridian colored by $2$ and longitude colored by $2$, and let $S$ denote a 2-sphere with two uncolored points and one point colored $2$,
then we obtain a map
\begin{equation}
\TV(K,2,2)=Z(T_2):\vs \rightarrow\vs.
\notag
\end{equation}
By using the trivalent graph basis in \cite{BHMV},
\begin{equation}
\vs=\text{Span}<b_1,b_2,b_3>,
\notag
\end{equation}
where $b_1,b_2,b_3$ are as in Figure \ref{basis3}.
\begin{figure}[h]
\includegraphics[width=0.8in,height=0.5in]{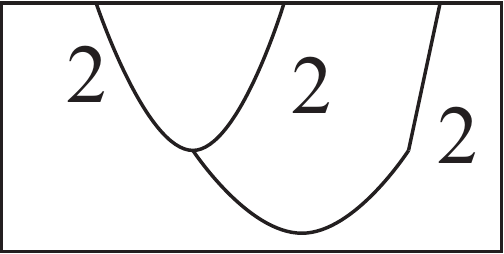}\hskip 0.2in
\includegraphics[width=0.8in,height=0.5in]{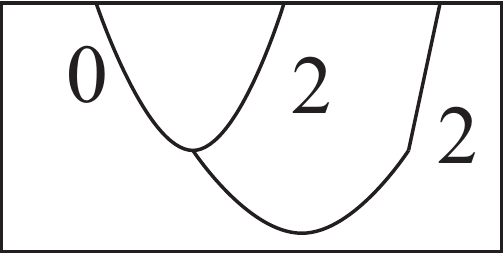}\hskip 0.2in
\includegraphics[width=0.8in,height=0.5in]{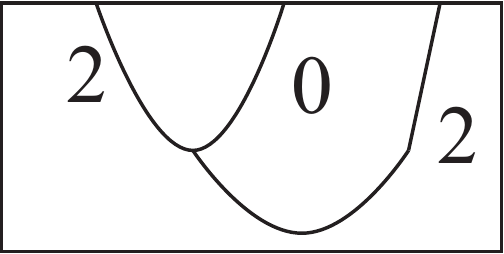}
\caption{A basis for $\vs$ where $S$ is a 2-sphere with two uncolored points and one point colored 2}
\label{basis3}
\end{figure}

With respect to this basis,
we have
\begin{equation}
TV(K,2,2)=\kappa^{-1}\eta
\begin{bmatrix}
  -(\frac{ A^8}{\Delta}+\frac{ A^{8k+8}\Delta \Tet}{\Theta^2}) & 0 & 0 \\
  0 & 0 & 0 \\
  0 & 0 & 0   
 \end{bmatrix} \ .
\notag
\end{equation}

By Proposition \ref{ker}, we also have
$TV(K,2,2)=
\begin{bmatrix}
  - \kappa^{-1}\eta(\frac{ A^8}{\Delta}+\frac{ A^{8k+8}\Delta \Tet}{\Theta^2}) 
 \end{bmatrix} \ .
$
This  last expression  lies in $\bz[A]$ for all $k$.
One has that:
\begin{eqnarray}
&&\Trace (TV(K,0,2))+\Delta\Trace (TV(K,2,2))\notag\\
&=&\kappa^{-1}\eta\Delta(1+A^8+\Delta A^{8k+8}) -\kappa^{-1}\eta\Delta(\frac{A^8}{\Delta}+\frac{A^{8k+8}\Delta \Tet}{\Theta^2}). \notag
\end{eqnarray}
Moreover, we used recoupling theory as in \cite{MV,KL, L1} to calculate the $2$-colored Jones polynomial of $k$-twist knot directly to obtain:
\begin{equation}
J(K,2)=-\frac{A^4}{\Delta}+ (1+\frac{\Delta^2\Tet}{\Theta^2 A^8})A^{8k}\notag
\end{equation}
We used Mathematica to verify that the two calculations agree for all $k$.

\subsection{The Turaev-Viro endomorphism of $6_2$ and quantum invariant of $S^3(6_2)$.}
In this section,
we will compute the Turaev-Viro endomorphism and quantum invariant of $S^3(6_2)$ when $A$ is a primitive $10$th root of unity
and verify that the trace of the Turaev-Viro endomorphism equals to quantum invariant.
By $6_2$, we mean the knot as pictured in \cite{CL}, which is the mirror image of the knot as pictured in \cite{L1,R2}.
A tangle $T$ for $S^3(6_2)$ is as in Figure \ref{sixtwo}.
\begin{figure}[h]
\includegraphics[width=1in,height=1.2in]{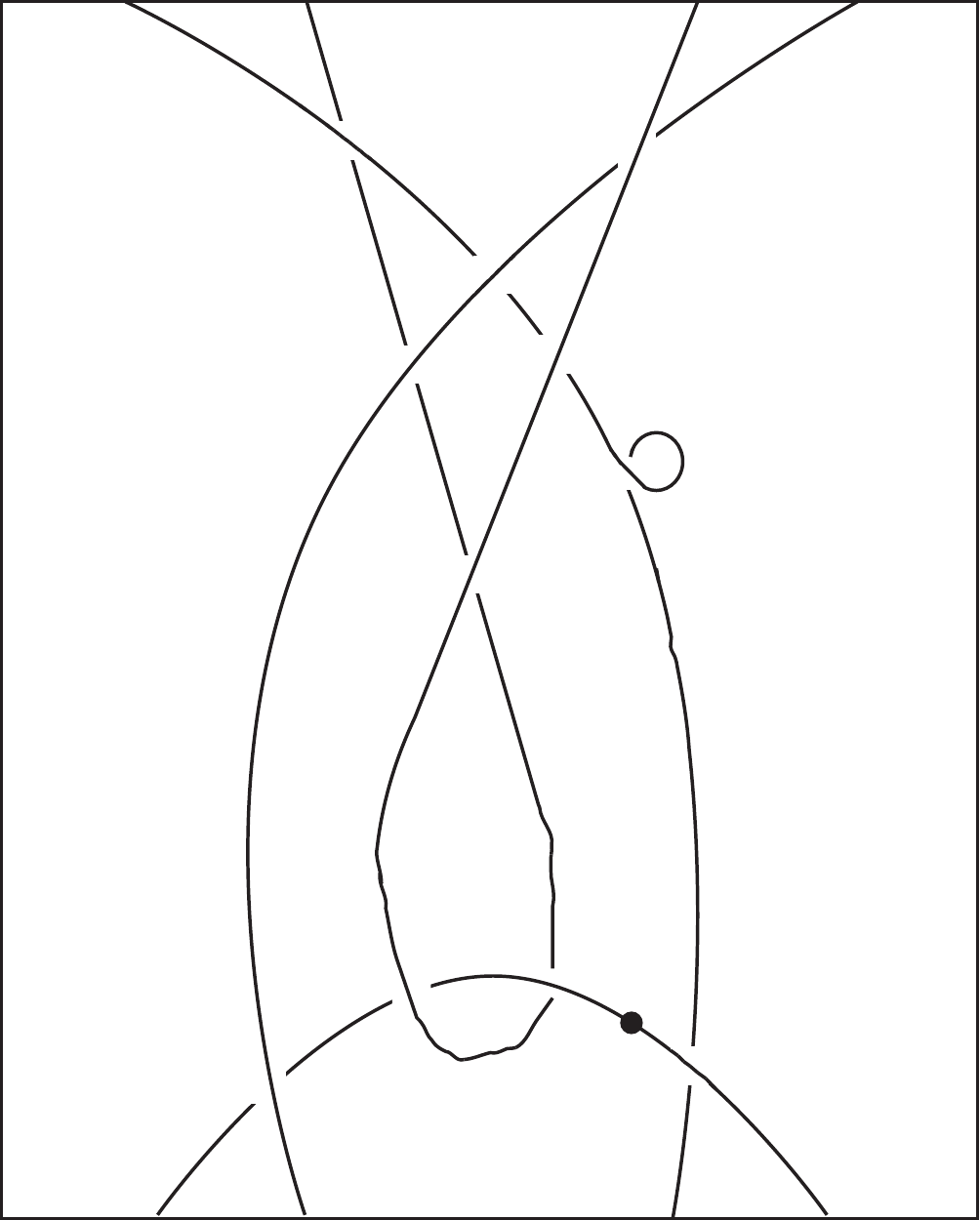}
\caption{Tangle for $S^3(6_2)$, with a choice for the position for the black dot.}
\label{sixtwo}
\end{figure}

So we obtain a map
\begin{equation}
\TV(6_2,0,0)=Z(T):\vs\rightarrow\vs.
\notag
\end{equation}
where $S$ is a 2-sphere with four uncolored points.
We use a trivalent graph basis in \cite{BHMV} for $\vs$ as in Figure \ref{basis}.
\begin{figure}[h]
\includegraphics[width=3in,height=4in]{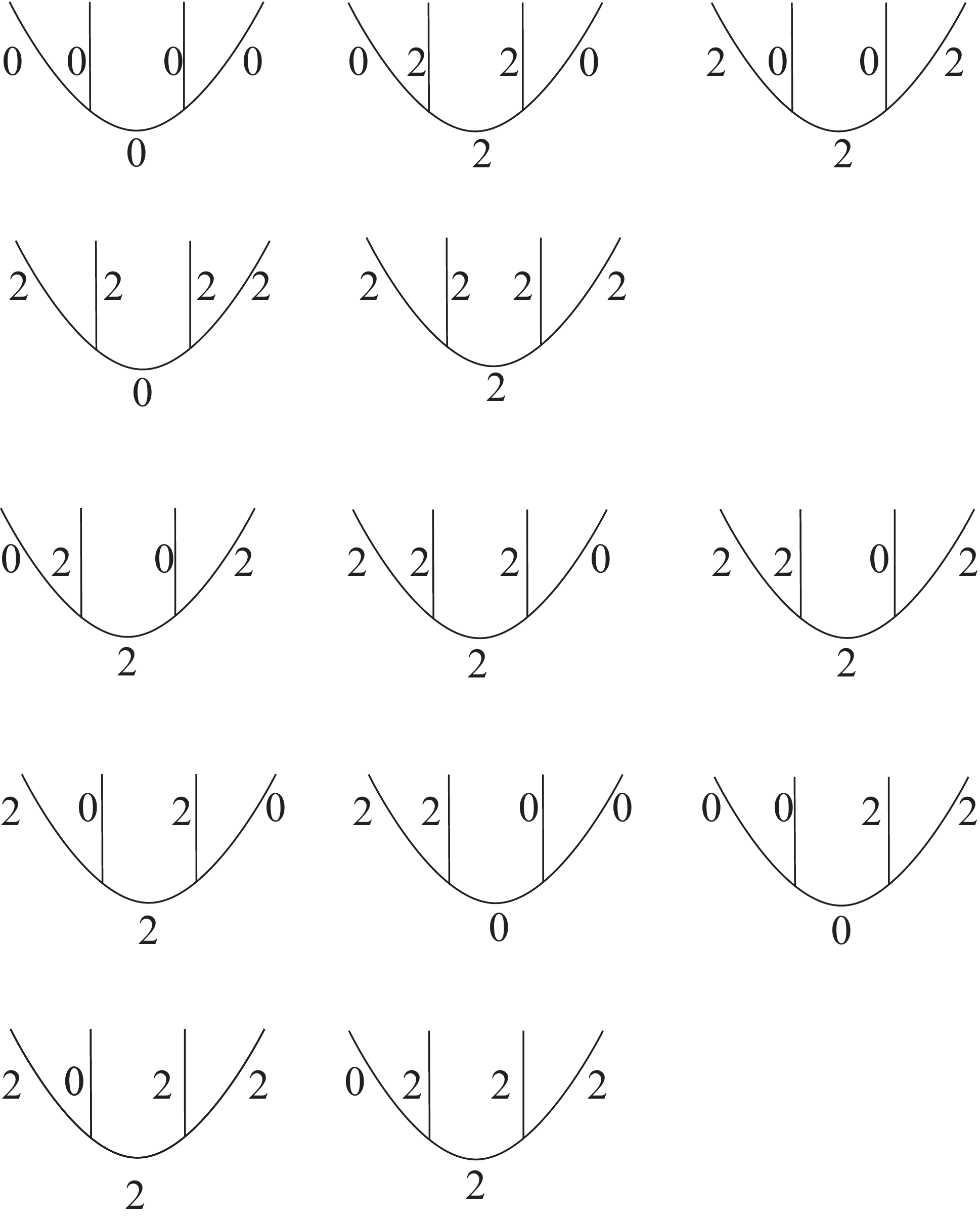}
\caption{A basis for $\vs$, where $S$ is a 2-sphere with four uncolored points.}
\label{basis}
\end{figure}

With respect to this basis,
we can obtain a $13\times 13$ matrix, which is in the strong shift equivalence class of the Turaev-Viro endomorphism.
However, by Proposition \ref{ker} 
 applied twice in succession,
it is enough to consider the minor 
given by the first  five rows and columns.  
We thus obtain a $5\times 5$ matrix:
\[
 \TV(6_2,0,0)= \kappa^{-1}\eta
 \begin{pmatrix}
  1 & 0 & \Delta^2 & 0 & 0 \\
  A^2 & 0 & \Delta(A^3+A)  & 0 & 0 \\
  0  & A^4  & 0 & \Delta & A^2\Theta  \\
  0 & \frac{A^8}{\Delta} & 0 & A^8 & \frac{A^8\Theta}{\Delta} \\
  0 & \frac{\Delta A^2}{\Theta} & 0 & \frac{\Delta^2}{\Theta} & \Delta(1-A^6+A^8+\frac{(A^4-A^6)\Delta\Tet}{\Theta^2})
 \end{pmatrix}.
\]

The Turaev-Viro polynomial (at $p=5$) is the characteristic polynomial of the above matrix, namely:
$$x^5+ \left(A^3+A-1\right) x^4+\left(-A^3-A^2-A\right)
   x^3+ \left(A^2+A+1\right) x^2+\left(A^3-A^2-1\right) x -A^3 .$$ 
We also note 
that
\begin{equation}
\Trace(\TV(6_2,0,0))=1-A-A^3.\notag
\end{equation} 
The left hand side of the first equation in Theorem \ref{TraceequalColoredJonesSum}, with $i=j=0$, and $K=6_2$ is by definition, the quantum invariant of $S^3(6_2)$. 
The right hand side is, by direct computation: 
$$\eta^2(J({6_2},0)\Delta_0+  J({6_2},2) \Delta_2= \eta^2( 1+\Delta(-A^{-2}+A^2-A^8-A^6))=1-A-A^3$$
Therefore, we verify a case of the first equation in Theorem \ref{TraceequalColoredJonesSum}.


\section*{Acknowledgements}
The second author  was partially supported by NSF-DMS-0905736. 
The first author was partially supported by a research assistantship funded by this same grant.
We would like to thank the referee for useful suggestions for improving the exposition.

\end{document}